\documentclass[12pt]{article}
\usepackage{amssymb, amsmath, amsthm, amscd}
\usepackage[dvips]{graphics}
\usepackage[utf8]{inputenc}
\usepackage[all,cmtip]{xy}
\usepackage{bbm}
\usepackage{eurosym}
\usepackage{enumitem}
\usepackage{setspace}
\usepackage{color}
\usepackage[usenames,dvipsnames,svgnames,table]{xcolor}
\usepackage[colorlinks=true,
            linkcolor=blue,
            urlcolor=blue,
            citecolor=blue]{hyperref}
\usepackage{filecontents}

\begin{document}

\newtheorem{lem}{Lemma}[section]
\newtheorem{pro}[lem]{Proposition}
\newtheorem{defi}[lem]{Definition}
\newtheorem{def/not}[lem]{Definition/Notations}

\newtheorem{thm}[lem]{Theorem}
\newtheorem{ques}[lem]{Question}
\newtheorem{cor}[lem]{Corollary}
\newtheorem{rem}[lem]{Remark}
\newtheorem{rqe}[lem]{Remarks}
\newtheorem{exa}[lem]{Example}
\newtheorem{exas}[lem]{Examples}
\newtheorem{obs}[lem]{Observation}
\newtheorem{corcor}[lem]{Corollary of the corollary}
\newtheorem*{ackn}{Acknowledgements}

\newcommand{\C}{\mathbb{C}}
\newcommand{\R}{\mathbb{R}}
\newcommand{\N}{\mathbb{N}}
\newcommand{\Z}{\mathbb{Z}}
\newcommand{\Q}{\mathbb{Q}}
\newcommand{\Proj}{\mathbb{P}}
\newcommand{\Rc}{\mathcal{R}}
\newcommand{\Oc}{\mathcal{O}}
\newcommand{\diff}{\textit{diff}}

\begin{center}

{\Large\bf The Dirichlet problem for the complex Hessian operator in the class $\mathcal{N}_m(H)$}

\end{center}

\begin{center}

{\large Ayoub El Gasmi}

\end{center}

%\author{El Gasmi Ayoub}

%\address{El Gasmi Ayoub\\Universite Ibn Tofail\\
%\email{gasmi1ayoub@gmail.com}
%\date{\today}
%\maketitle
%\setcounter{tocdepth}{1}

\vspace{1ex}

\noindent{\small{\bf Abstract.}
We prove that, in a $m$-hyperconvex domain in $\mathbb{C}^{n},$ if a non-negative Borel measure is dominated by a complex Hessian measure, then it is a complex Hessian measure of a function in the class $\mathcal{N}_m(H)$. This is an extension of P. $\hbox{\AA}$hg, U. Cegrell, R. Czy$\dot{z}$ and P.H. Hiep's result in \cite{ACCH}.

\vspace{1ex}

\section*{Introduction}\label{section:introduction}
The subsolution theorem due to Kolodziej \cite{Ko2} says that if the Dirichlet problem (\ref{1114}) with nonnegative Borel measure $d\mu$ in a strictly pseudoconvex domain $\Omega\subset\mathbb{C}^{n}$ and continuous
boundary data $\varphi\in C(\partial\Omega)$:
\begin{equation}\label{1114}
\left\{
     \begin{array}{lll}
       u\in\mathrm{PSH}(\Omega)\cap L^{\infty}(\Omega), \\
        (dd^cu)^n=d\mu,\\
        u=\varphi       \;\;\;\; \hbox{on}\;\;\;\; \partial\Omega,
     \end{array}
   \right.
\end{equation}
has subsolution, then $(\ref{1114})$ is solvable. Nguc Cuong Nguyen showed in \cite{Ngo} that the subsolution theorem for the complex Hessian equation is still true, he proved that the existence of a subsolution is a necessary and sufficient
condition for the existence of a bounded solution to the Dirichlet problem for the complex Hessian equation. 
In the other hand, $\hbox{\AA}$hg, Cegrell, Czyz and Hiep obtained in \cite{ACCH} a generalization of Kolodziej's subsolution theorem. More precisely, they proved that if a non-negative Borel measure is dominated by a complex Monge-Amp\`{e}re measure, then it is a complex Monge-Amp\`{e}re measure. 
\par Throughout this paper it is always assumed that $\Omega$ is a bounded $m$-hyperconvex domain (see the next section for the definition of $m$-hyperconvex domain). The purpose of this paper is to rewrite the work of \cite{ACCH} in the case of the complex Hessian equation. With notations introduced in the next section, our main result is the following theorem.\\

\textbf{Main Theorem.}
\textit{Assume that  $\mu$ is a non-negative measure. If there exists a function $\omega \in \mathcal{E}_m$  with $\mu \leq H_m(\omega),$ then for every function
$H\in\mathcal{E}_m\cap \mathcal{MSH}_m(\Omega),$ there exists a function $u \in \mathcal{E}_m$, $\omega+H\leq u\leq H$ such that $H_m(u)=\mu.$ In particular, if $\omega \in \mathcal{N}_m$, then $u \in \mathcal{N}_m(H).$}\\

Recently, in \cite{HP}, Vu Viet Hung and Nguyen Van Phu proved this result in the particular case when $\mu(\Omega)<+\infty$ and $H=0$.

This paper is organized as follows. In Section \ref{7777} we recall some basic properties of $m$-subharmonic functions, the complex Hessian operator and Cegrell classes, and we introduce the class $\mathcal{N}_m(H)$. The main sources are \cite{Chi2}, \cite{Blo1}. For the proofs we refer the reader to these references modulo some simple modifications. In section \ref{77771}, we give the comparison principle for the class $\mathcal{N}_m(H)$. In section \ref{77772} we study the Dirichlet Problem for the complex Hessian operator with continuous bondary data. The last section is devoted to the proof of the main Theorem.\\
\section{Preliminaries}\label{7777}
\subsection{m-subharmonic functions}
In this section, we give some basic properties of admissible functions for the complex Hessian equation. Such functions are called $m$-subharmonic ($m$-sh), they are subharmonic and non-smooth in general. 
Let $1\leq m\leq n.$ Put  $$S_m(\lambda)=\sum_{1\leq j_1<\cdots< j_m\leq n}\lambda_{j_{1}}\ldots\lambda_{j_{m}},$$
called the symetric  function of $\mathbb{R}^{n}$ of degree $m$, which can be determined by
                 $$(\lambda_1+t)\cdots(\lambda_n+t)=\sum_{m=0}^{n}S_m(\lambda)t^{n-m}, \;\;\;\; \hbox{with}\;\;\; t\in \mathbb{R}.$$
We denote $\Gamma_m$ the closure of the  connected component of $\{\lambda\in \mathbb{R}^n: S_m(\lambda)>0\}$ containing $(1,\cdots,1).$ Let $t\geq 0$, we have
 \begin{eqnarray*}
  \Gamma_m= \{\lambda\in \mathbb{R}^n: S_m(\lambda_1+t,\cdots\lambda_n+t)\geq0\} = \{\lambda\in \mathbb{R}^n: \sum_{p=0}^{m}S_m(\lambda)t^{n-m}\geq0\}= \displaystyle\bigcap_{p=0}^{m}\{S_p\geq0\},&&
\end{eqnarray*}
Note that, $\Gamma_n\subset \Gamma_{n-1}\subset \cdots\subset \Gamma_1,$ and by the results in \cite{Ga}, $\Gamma_m$ is convex in $\mathbb{R}^{n}$ and $(S_m)^{\frac{1}{m}}$ is concave in $\Gamma_m$, and by the Maclaurin inequality
$$\left(
    \begin{array}{c}
      n \\
      m \\
    \end{array}
  \right)^{\frac{-1}{m}}(S_m)^{\frac{1}{m}}\leq\left(
    \begin{array}{c}
      n \\
      p \\
    \end{array}
  \right)^{\frac{-1}{p}}(S_p)^{\frac{1}{p}}, \;\;\;\;\;\;\;\;\;\;\; \forall\; 1\leq p\leq m\leq n.
$$
Let $\mathcal{H}$ be the real vector space of complex hermitian matrix $n\times n$, For any  $A\in \mathcal{H}$, let $\lambda(A)=(\lambda_1,\cdots,\lambda_n)\in \mathbb{R}^{n}$ be the vector of the eigenvalues
of $A$. We set
                                         $$\tilde{S}_m(A)=S_m(\lambda(A)),$$
\hbox{and\;define\;the\;cone}
$$\tilde{\Gamma}_m:=\{A\in \mathcal{H}: \lambda(A)\in \Gamma_m\}=\{A\in \mathcal{H}: \tilde{S}_k(A)\geq0, \forall \;1\leq k\leq m\}.$$
Let $\alpha$ be a $(1,1)$-form such that
$$\alpha=\frac{i}{2}\sum_{j,k}a_{j\bar{k}}dz_j\wedge d\bar{z}_k,$$
where $A=(a_{j\bar{k}})$ is a complex hermitian matrix. After diagonalizing the matrix $A$ we see that
           $$\alpha^m\wedge \beta^{n-m}=\frac{m!(n-m)!}{n!}\tilde{S}_m(\alpha)\beta^{n}$$
where $\beta=dd^{c}\vert z\vert^{2}$ is the standard K\"{a}hler form in $\mathbb{C}^n$, The last equality allows us to define
   $$\breve{\Gamma}_m:=\{\alpha\in \mathbb{C}_{1,1}:\; \alpha\wedge \beta^{n-1}\geq0,\alpha^2\wedge \beta^{n-2}\geq0,\cdots,\alpha^m\wedge \beta^{n-m}\geq0\},$$
where $\mathbb{C}_{(1,1)}$ is the space of real $(1,1)$-forms with constant coefficients in $\mathbb{C}^n$. Note that a $(1,1)$-form belonging to $\breve{\Gamma}_m$ is called $m$-positive and if $T$ is a current  of bidegree $(n-k, n-k)$, with $k \leq m$. Then $T$ is called $m$-positive if for all $m$-positive
$(1,1)$-forms $\alpha_1,...,\alpha_{k}$ we have 
$$\alpha_1\wedge\alpha_2\wedge\ldots\wedge\alpha_k\wedge T\geq 0.$$
\par Let $M:\mathbb{C}_{(1,1)}\longrightarrow \mathbb{R}$ be the polarized form of  $\widetilde{S}_m$ (i.e $M$ is is linear in every variable,
symmetric and $M(\alpha,\cdots,\alpha)=\tilde{S}_m(\alpha)$, for any $\alpha \in \mathbb{C}_{(1,1)}).$
by the Garding inequality  (see \cite{Ga}) we have
$$M(\alpha_1,\cdots,\alpha_m)\geq\tilde{S}_m(\alpha_1)^{1/m},\ldots,\tilde{S}_m(\alpha_m)^{1/m},\;\; \alpha_1,\ldots, \alpha_m \in \breve{\Gamma}_{m}.$$
\begin{pro}
If $\alpha_1,\ldots, \alpha_p \in \breve{\Gamma}_m, 1\leq p\leq m$, then $\alpha_1\wedge\alpha_2\wedge\ldots\wedge\alpha_p\wedge\beta^{n-m}\geq 0.$
\end{pro}
\begin{defi}
Let $\Omega$ be a bounded domain in $\mathbb{C}^{n}$. Then $\Omega$ is called $m$-hyperconvex if there exists a continuous $m$-sh function $\varphi:\; \Omega \rightarrow \mathbb{R}^{-}$ such that $\{\varphi<c\}\Subset\Omega$, for every $c<0.$
\end{defi}
In connection with the results above, we give the definition of a $m$-sh function due to Blocki (see \cite{Blo1}).
\begin{defi}
Let $u:\Omega\longrightarrow \mathbb{R}\cup\{-\infty\}$  be a subharmonic function in $\Omega.$
\begin{itemize}
  \item[$(i)$] If $u\in C^2(\Omega)$, then $u$ is $m$-sh if the form $dd^cu$ belongs pointwise to $\breve{\Gamma}_m$ .
  \item[$(ii)$] For non-smooth case, $u$ is called $m$-sh if the  inequality
   $$dd^cu\wedge\alpha_1\ldots \wedge \alpha_{m-1}\wedge \beta^{n-m}\geq0,\;\;\;\;\;\;\alpha_1,...,\alpha_{m-1}\in\breve{\Gamma}_m,$$
holds in the weak sense of currents in $\Omega$.
\end{itemize}
\end{defi}
We denote by $\mathcal{SH}_m(\Omega)$ the set of all $m$-sh functions in $\Omega$. Blocki observed that up to a point pluripotential theory can be adapted to $m$-sh functions.  We recall some
properties of $m$-subharmonic  functions.
\begin{pro}(\cite{Blo1}).
\begin{description}
  \item[1).] $\mathcal{PSH}=\mathcal{SH}_{n}\subset\mathcal{SH}_{n-1}\subset\ldots\subset\mathcal{SH}_1\subset \mathcal{SH}$.
  \item[2).] If $u, v\in \mathcal{SH}_m(\Omega)$ then $\lambda u+ \nu v \in \mathcal{SH}_m(\Omega)$, $\forall \lambda, \mu\geq 0.$
  \item[3).] Let $[u_j]_{j\in\mathbb{N}}\subset$ be a decreasing sequence of $m$-subharmonic functions in $\Omega$ that converges to a function $u$, then $u\in\mathcal{SH}_m(\Omega).$
  \item[4).] If $u\in \mathcal{SH}_m(\Omega)$ and $f$ is a  convex increasing function, then $f\circ u\in \mathcal{SH}_m(\Omega).$
  \item[5).] If $u\in \mathcal{SH}_m(\Omega)$, then the standard regularization $u\ast\rho_{\epsilon}\in\mathcal{SH}_m(\Omega_{\epsilon})$, where $\Omega_{\epsilon}:=\{z\in\Omega\; |\; dis(z,\partial\Omega>\epsilon) \}$, for $0<\epsilon\ll1.$
  \item[6).] If $[u_j]\subset\mathcal{SH}_m(\Omega)\cap L_{loc}^{\infty},$ then $(\sup u_j)^{*}\in\mathcal{SH}_m(\Omega)$, where $\theta^{*}$ denotes the upper semicontinuous regularisation of $\theta.$
\end{description}
\end{pro}
For locally bounded $m$-sh functions, we can inductively define a closed nonnegative current (following Bedford and Taylor for plurisubharmonic functions).
$$dd^cu_1\wedge\ldots\wedge dd^cu_k\wedge\beta^{n-m}:=dd^c(u_1dd^cu_2\wedge\ldots\wedge dd^cu_k\wedge\beta^{n-m}),$$
where $u_1,\ldots, u_k\in \mathcal{SH}_m(\Omega)\cap L_{loc}^{\infty}(\Omega).$
In particular, we define the nonnegative Hessian measure for a function $u\in \mathcal{SH}_m(\Omega)\cap L_{loc}^{\infty}(\Omega)$,
$$H_m(u)=(dd^cu)^m\wedge\beta^{n-m}.$$
\subsection{Cegrell's classes and Approximation of $m$-sh functions}
The following classes of $m$-sh functions were introduced by Chinh in \cite{Chi1} and \cite{Chi2}.
\begin{defi}
\begin{itemize}
\item We denote $\mathcal{E}^0_m$ the class of bounded functions that is belong to $\mathcal{SH}_{m}^{-}(\Omega)$ such that
$\displaystyle\lim_{z\rightarrow \xi}u(z)=0$, $\forall\xi\in\partial\Omega$ and $\displaystyle\int_{\Omega}H_m(u)<+\infty.$
\item Let $u \in \mathcal{SH}_{m}^{-}(\Omega)$, we say that $u$  belongs to $\mathcal{E}_m(\Omega)$ (shortly $\mathcal{E}_m$) if for each $z_0 \in \Omega$, there exist an open neighborhood $U \subset \Omega$ of $z_0$ and a decreasing sequence $[u_j] \subset \mathcal{E}_m^0$ such that 
  $u_j \downarrow u$ on $U$ and $\displaystyle\sup_j\int_\Omega H_m(u_j)<+\infty.$
\item We denote by $\mathcal{F}_m(\Omega)$ (or $\mathcal{F}_m$) the class of functions $u \in \mathcal{SH}_{m}^{-}(\Omega)$ such that there exists a sequence $(u_j) \subset \mathcal{E}_m^0$  decreasing to $u$ in $\Omega$ and $\displaystyle\sup_j\int_\Omega H_m(u_j)<+\infty.$
\item For every $p\geq 1$, $\mathcal{E}^p_m$ denote the class of functions $\psi \in \mathcal{SH}^{-}_m(\Omega)$ such that there
exists a decreasing sequence  $[\psi_j]\subset\mathcal{E}^0_m$ such that $\displaystyle\lim_{j\rightarrow +\infty}\psi_j(z)=\psi(z),$ and  $\sup_j\displaystyle\int_{\Omega}(-\psi_{j})^pH_m(\psi_j)<+\infty.$\\
\;\;\; If moreover $\displaystyle\sup_j\int_{\Omega}H_m(\psi_j)<+\infty$ then, by definition, $\psi\in \mathcal{F}^p_m.$
\end{itemize}
\end{defi}
\begin{pro}\label{29}
Suppose $u_1,...,u_m \in\mathcal{F}_m$ and $h\in\mathcal{SH}^{-}_m(\Omega)$, if $u_{1}^{j},...,u_{m}^{j}$ is a sequence of functions in $\mathcal{E}_{m}^{0}$ decreasing to $u_1,...,u_m$ respectively, as $j\rightarrow+\infty$ and $\displaystyle\int hdd^{c}u^{1}\wedge...\wedge dd^{c}u^{n}<+\infty,$ then we have the following
\begin{equation}\label{1500}
\lim_{j\rightarrow+\infty}\int hdd^{c}g^{j}_{1}\wedge...\wedge dd^{c}g^{j}_{m}\wedge\beta^{n-m}=\int hdd^{c}u_{1}\wedge...\wedge dd^{c}u_{m}\wedge\beta^{n-m}.
\end{equation} 
\begin{equation}\label{1501}
hdd^{c}g^{j}_{1}\wedge...\wedge dd^{c}g^{j}_{m}\wedge\beta^{n-m} \;converges\; weakly\; to\; hdd^{c}u_{1}\wedge...\wedge dd^{c}u_{m}\wedge\beta^{n-m}.
\end{equation}
\end{pro}
\begin{proof}
It follows from \cite[Theorem 3.11]{Chi2} that $$dd^{c}g^{j}_{1}\wedge...\wedge dd^{c}g^{j}_{m}\wedge\beta^{n-m} \hbox{\;converges\; weakly\;to}\; hdd^{c}u_{1}\wedge...\wedge dd^{c}u_{m}\wedge\beta^{n-m}.$$ Furtheremore, since  $\Omega$ is open then
$$\displaystyle\int dd^{c}u_{1}\wedge...\wedge dd^{c}u_{m}\wedge\beta^{n-m}\leq\displaystyle\liminf_{j\rightarrow+\infty}\int dd^{c}g^{j}_{1}\wedge...\wedge dd^{c}g^{j}_{m}\wedge\beta^{n-m}<+\infty.$$
If we suppose first that $h$ $\in$ $\mathcal{E}_{m}^{0}$, then by \cite[Theorem 3.13]{Chi2} we have
$\displaystyle\lim_{j\rightarrow+\infty}\int hdd^{c}g^{j}_{1}\wedge...\wedge dd^{c}g^{j}_{m}\wedge\beta^{n-m}=\displaystyle\int hdd^{c}u_{1}\wedge...\wedge dd^{c}u_{m}\wedge \beta^{n-m}.$
Suppose now that $h\in\mathcal{SH}_{m}^{-}(\Omega),$ then it follows from \cite[Theorem 3.1]{Chi2} that for each $j$, we can choose
$h_j$ $\in$ $\mathcal{E}_{m}^{0}\cap\mathrm{C}(\Omega)$ decreasing to $h$. So, to finish the proof of (\ref{1500}) it suffices to follow the argument in \cite[Proposition 5.1]{Chi2}. For (\ref{1501}), take $\xi\in C^{\infty}_{0}(\Omega)$, then $h\xi$ is  upper semicontinuous. Thus,
$$\liminf_{j\rightarrow+\infty}\int_{\Omega}(-h)\xi dd^{c}g^{1}_{j}\wedge 
...\wedge dd^{c}g^{m}_{j}\wedge\beta^{n-m}\geq\int_{\Omega}(-h)\xi dd^{c}u^{1}\wedge...\wedge dd^{c}u^{m}\wedge\beta^{n-m}.$$
Let $\nu$ be the weak limit of $hdd^{c}g^{1}_{j}\wedge ...\wedge dd^{c}g^{m}_{j}\wedge\beta^{n-m}$, then 
$\nu\geq -hdd^{c}u^{1}\wedge...\wedge dd^{c}u^{m}\wedge\beta^{n-m}$, On the other hand,
$$\int_{\Omega}d\nu\leq \lim_{j\rightarrow+\infty}\int_{\Omega}-hdd^{c}g^{1}_{j}\wedge ...\wedge dd^{c}g^{m}_{j}\wedge\beta^{n-m}=\int_{\Omega}-hdd^{c}u^{1}\wedge...\wedge dd^{c}u^{m}\wedge\beta^{n-m}.$$
Therefore $\nu=-hdd^{c}u^{1}\wedge...\wedge dd^{c}u^{m}\wedge\beta^{n-m}.$
\end{proof}
\begin{thm}\label{40}\cite[Proposition 3.3]{HP}.
If $u_1,...,u_m$ $\in$ $\mathcal{F}_m$ and $h$ $\in$ $SH_m(\Omega)^{-}_m$. Then
$$
\displaystyle\int-hdd^{c}u_{1}\wedge...\wedge dd^{c}u_{m}\wedge \beta^{n-m}\leq\displaystyle\prod_{k=1}^{m}\left(\int -hH_m(u_{k})\right)^{1/m}.$$
\end{thm}
\begin{cor}\label{570}
Let $u_1,...,u_m$ $\in$ $\mathcal{F}_m$. Then $$\displaystyle\int dd^{c}u_{1}\wedge...\wedge dd^{c}u_{m}\wedge\beta^{n-m}\leq\displaystyle\prod_{k=1}^{m}\left(\int H_m(u_{k})\right)^{1/m}.$$
\end{cor}
\begin{lem}\label{28}
\begin{description}
  \item[(1)] Let $u, u_k, v \in \mathcal{E}_m,$ $k=1,...,m-1,$ with $u\geq v$ on $\Omega$ and set $T=dd^cu_1\wedge...dd^cu_{m-1}\wedge\beta^{n-m}.$ Then
$$\chi_{\{u=-\infty\}}dd^cu\wedge T\leq \chi_{\{v=-\infty\}}dd^cv\wedge T.$$
In particular, if $u$, $v \in \mathcal{E}_m$ are such that  $u\geq v,$ then we have that
$$\int_{A} H_m(u)\leq\int_{A} H_m(u),\; \hbox{for\; every}\; \hbox{m-polar}\; set\; A\subset\Omega.$$
\item[(2)] Let $\mu$ be a positive measure which vanishes on all $m$-polar subsets of $\Omega$. Suppose that
$u, v \in\mathcal{E}^0_m$ such that $H_m(u)\geq\mu$ and $H_m(v)\geq\mu$. Then $H_m(\max(u,v))\geq\mu.$
\item[(3)] Suppose $u_1,u_2,...,u_m$ $\in$ $\mathcal{E}_m$. Then for every $m$-polar $A\subset \Omega$ we have
$$\int_{A} dd^{c}u_{1}\wedge...\wedge dd^{c}u_{m}\wedge\beta^{n-m}\leq\left(\int_{A} H_m(u_{1})\right)^{1/m}...\left(\int_{A} H_m(u_{m})\right)^{1/m}.$$
\end{description}
\end{lem}
\begin{proof}
For the first inequality in $\textbf{(1)}$, $\textbf{(2)}$ and $\textbf{(3)}$ See \cite[Proposition 5.2 $\&$ Lemma 5.6 ]{HP}. For the second inequality in $\textbf{(1)}$, it follows from the first one in $\textbf{(1)}$ that we have
$$
\int_{A} H_m(u)= \int_{A\cap\{u=-\infty\}} H_m(u)= \int_{A}\chi_{\{u=-\infty\}} H_m(u)\leq \int_{A}\chi_{\{v=-\infty\}} H_m(v)=\int_{A} H_m(v).
$$
\end{proof}
\begin{lem}\label{34}
Assume that $v\in\mathcal{E}^0_m$ and $\omega \in \mathcal{E}_m$ such that $H_m(v)=H_m(\omega).$ Then $v\geq \omega.$
\end{lem}
\begin{proof}
Let $u=\max(v,\omega),$ since $u\in\mathcal{E}^0_m$ and
$$
 H_m(u) = H_m(\max(v,\omega))\geq \chi_{\{v<\omega\}}H_m(\omega) +\chi_{\{v\geq\omega\}}H_m(v)\geq \chi_{\{v<\omega\}}H_m(v)+\chi_{\{v\geq\omega\}}H_m(v)\geq H_m(v).
$$
Then $u\leq v$, so $u=v$, Hence $v\geq \omega.$
\end{proof} 
\begin{defi}
Let $\Omega$ be an open set in $\mathbb{C}^{n}$, a function $u\in\mathcal{SH}_m(\Omega)$ is called $m$-maximal if $v\in\mathcal{SH}_m(\Omega)$, $v\leq u$ outside a compact subset of $\Omega$ implies that $v\leq u$ in $\Omega$
\end{defi}
In \cite{Blo1}, Blocki proved that a $m$-maximal functions $u$ in $\mathcal{E}_{m}$ are precisely the functions with $H_{m}(u)=0.$
Now we come to Characterize the class of $m$-sh function with the boundary values. Let $[\Omega_j]$ be the fundamental increasing sequence of strictly $m$-pseudoconvex subsets of $\Omega$ (that means that for each $j$ there exists a smooth strictly $m$-subharmonic function $\rho$ on some open
neighborhood $\Omega^{\prime}$ of $\Omega_j$ such that $\Omega_j:= \{z\in\Omega^{\prime} / \rho(z) < 0\})$ with $\Omega_j \Subset \Omega_{j+1}$ and $\displaystyle\bigcup^{\infty}_{j=1}\Omega_j=\Omega.$
\begin{defi}\label{120}
Let $u \in \mathcal{SH}^{-}_{m}(\Omega)$ and let $[\Omega_j]$ be a fondamental sequence. Let  $u^j$ be
 the function $$u^j=\sup\left\{\phi \in \mathcal{SH}_{m}(\Omega):\; \phi_{|_{\Omega\setminus\Omega_j}}\leq u\right\}\in\mathcal{SH}_{m}(\Omega),$$
and define $\widetilde{u}:=(\displaystyle\lim_{j\rightarrow+\infty}u^j)^{*},$ called the smallest $m$-maximal $m$-sh majorant of $u$. 
\end{defi}
Definition \ref{120} implies that $u\leq u^{j}\leq u^{j+1}$, therefore  $\displaystyle\lim_{j\rightarrow+\infty}u^j$ exists quasi-everywhere on $\Omega$ (i.e exept in an m-polar set), hence, $\widetilde{u}\in\mathcal{SH}_{m}(\Omega).$ Moreover, if $u\in \mathcal{E}_{m}$ then by [\cite{Chi1}, Theorem 1.7.5.] $\widetilde{u}\in\mathcal{E}_{m}$ and by \cite{Blo1} it is $m$-maximal on $\Omega.$ Let $u, v \in\mathcal{E}_{m}$ and $\alpha\in \mathbb{R}$, $\alpha\geq0,$ then we have that $\widetilde{u+v}\geq \widetilde{u}+\widetilde{v}$, $\widetilde{\alpha u}=\alpha\widetilde{u},$ and if moreover $u\leq v$ then $\widetilde{u}\leq\widetilde{v}.$ It follows from \cite{Blo1} that $\mathcal{E}_{m}\cap \mathcal{MSH}_{m}(\Omega)=\{u\in\mathcal{E}_{m}:\;\widetilde{u}=u\},$ where $\mathcal{MSH}_{m}(\Omega)$ is the family of $m$-maximal functions in $\mathcal{SH}_{m}(\Omega)$.
\par Set $\mathcal{N}_m:=\{u \in \mathcal{E}_{m}:\; \widetilde{u}=0\}.$
Then we have that $\mathcal{N}_m$ is a convex cone and that it is precisely the set of functions in $\mathcal{E}_{m}$ with smallest $m$-maximal $m$-sh majorant identically zero. Note also that 
$$\mathcal{E}^{0}_m\subset \mathcal{F}_m\subset \mathcal{N}_m\subset \mathcal{E}_m.$$
\begin{defi}
Let $\mathcal{K}_m \in \{\mathcal{E}_{m}^{0}, \mathcal{F}_m,\mathcal{F}_m^{p}, \mathcal{N}_m\}$ and $H \in \mathcal{E}_m$. We say that a $m$-sh function $u$ defined on $\Omega$ belongs to $\mathcal{K}_m(\Omega,H)$ (shortly $\mathcal{K}_m(H)$) if there exists $\varphi \in \mathcal{K}_m$ such that
                                                    $H\geq u\geq \varphi+H.$
\end{defi}
Note that if $H=0,$ then $\mathcal{K}_m(H)=\mathcal{K}_m.$
Let $H \in\mathcal{E}_m,$ we define
  $$\mathcal{N}^a_m:=\{u \in \mathcal{N}_{m}:\; H_{m}(u)(P)=0,\; \forall P \;m\hbox{-polar\; set}\},$$
$$\hbox{and}\;\;\mathcal{N}^a_m(H):=\left\{u \in \mathcal{E}_m\; /\; \exists \varphi \in \mathcal{N}^a_m \; \hbox{such that} \;\;H\geq u\geq \varphi+H\right\}.$$
 The following approximation proposition is a consequence of \cite[Theorem 3.1]{Chi2}.
\begin{pro}\label{0}
Assume that  $H \in \mathcal{E}_m$ and $u\in \mathcal{SH}_m(\Omega)$ such that $u\leq H.$ Then there exists a decreasing sequence $[u_j] \subset \mathcal{E}_m^{0}(H)$ that converges pointwise to $u$ on $\Omega$, as $j\rightarrow+\infty$. Moreover, if $H \in \mathcal{SH}_m(\Omega)\cap C(\overline{\Omega}),$ then $[u_j]$ can be chosen such that $u_j\in\mathcal{E}_m^{0}(H)\cap C(\overline{\Omega}).$
\end{pro}
\begin{proof}
Let  $u\in \mathcal{SH}_m(\Omega)$ and $H \in \mathcal{E}_m$, by \cite[Theorem 3.1]{Chi2}, there exists a decreasing sequence $[\psi_j]\subset\mathcal{E}_{m}^{0}\cap C(\overline{\Omega})$ that converges pointwise to $u$ as $j\rightarrow+\infty$. Put $v_j=\max(u,\psi_j+H)$, so $v_j \in \mathcal{E}_{m}^{0}(H)$ and decreases pointwise to
$u$ as $j$ tends to $+\infty$, so the first statement is completed.
\par Now let $H \in \mathcal{SH}_m(\Omega)\cap C(\overline{\Omega})$ and $\varphi \in \mathcal{E}_m^0\cap C(\overline{\Omega}).$
We choose the fondamental sequence $[\Omega_j]$ of $\Omega$ such that for each $j \in \mathbb{N}$ we have $\varphi\geq -\frac{1}{2j^2}$ on $\Omega\setminus\Omega_j.$ Let $[v_j]\subset\mathcal{SH}_{m}(\Omega_j)\cap C^\infty(\Omega),$ be a decreasing sequence that converges pointwise to $u$ as $j$ tends to $+\infty$ and $v_j\leq H +\frac{1}{2j}$ on $\Omega_{j+1}.$ Set
       $$u^{'}_{j}=\left\{
     \begin{array}{ll}
       \max\{v_{j}-\frac{1}{j},j\varphi+H\} \ \ $on$ \ \  \Omega_{j}, \\
        j\varphi+H \ \ \ \ \ \ \ \ \ \ \ \ \ \ \ \ \  \ \ \ $on$\ \  \Omega\setminus\Omega_{j}.
     \end{array}
   \right.$$
Then $[u_{j}^{'}] \subset \mathcal{E}_m^{0}(H)\cap C(\overline{\Omega})$ and converges pointwise to $u$ as $j$ tends to $+\infty$. Set $u_j= \displaystyle\sup_{j\leq k}u_{k}^{'}.$ $[u_{j}^{'}]$ satisfies:
$$u_{j}^{'}+\frac{1}{j}\geq(v_{j+1}-\frac{1}{j+1})+\frac{1}{j+1}\;\; \hbox{and}\;\;
u_{j}^{'}+\frac{1}{j}\geq (j+1)\varphi+H+\frac{1}{j+1}.$$
  $$\hbox{Hence} \;\;\;\; u_{j}^{'}+\frac{1}{j} \geq \max\left\{v_{j+1}-\frac{1}{j+1},(j+1)\varphi+H\right\}= u_{j+1}^{'}+\frac{1}{j+1}.\;\;\;\;\;\;\;\;\;\;\;\;\;\;\;\;\;\;\;\;\;\;\;\;\;\;\;\;$$
Then, for each $j \in \mathbb{N}$ fixed,
  $\omega_{m}:=\left[\max(u_{j}^{'}, u_{j+1}^{'},\cdots, u_{m-1}^{'},u_{m}^{'}+\frac{1}{m})\right]^{\infty}_{m=j}$ decreases on $\Omega$ to $u_j$, as $m\rightarrow +\infty$ and since $\omega_m \in \mathcal{SH}_m(\Omega),$ then $u_j \in \mathcal{SH}_m(\Omega).$ Hence $u_j$ is upper semicontinuous, on the other hand, since $u_{k}^{'} \in  C(\overline{\Omega})$, then  $u_{k}^{'}$ is an lower semicontinuous, so is $u_j=\displaystyle\sup_{j\leq k}u_{k}^{'}$, this implies that $u_j$ is continuous on $\overline{\Omega}$. Moreover, $[u_j]$ is decreasing and converges pointwise to $u$ as $j\rightarrow+\infty.$ And the proof is completed.
\end{proof}
\subsection{Convergence in $m$-capacity}
Let $E \subset \Omega$ be a Borel subset. The $C_m$-capacity the $\tilde{C}_m$-capacity of $E$ with respect
to $\Omega$ are defined by
$$C_{m}(E)=C_{m}(E,\Omega)=\sup\left\{\int_{E}H_m(\theta)\;,\;\theta\in \mathcal{SH}_{m}(\Omega), -1\leq \theta\leq 0 \right\}.$$
$$\tilde{C}_{m}(E)=\tilde{C}_{m}(E,\Omega)=\sup\left\{\int_{E}H_{m-1}(\theta)\;,\;\theta\in \mathcal{SH}_{m}(\Omega), -1\leq \theta\leq 0 \right\}.$$
\par Let $[u_s]\subset \mathbb{N}$ be real-valued borel measurable function defined on $\Omega$,
$[u_s]$ is said to converges to $u$ in $\tilde{C}_{m}$-capacity, as $s\rightarrow+\infty$ if for every compact subset $K$ of $\Omega$ and every $\varepsilon\geq  0$ it holds that
$$\displaystyle\lim_{s\rightarrow+\infty}\tilde{C}_{m}(\{z\in K: |u_s(z)-u(z)|>\varepsilon\})=0.$$
\begin{thm}\cite[Theorem 3.10]{HP}.
Let $u_s,v_s,w\in \mathcal{E}_m$ be such that $u_s,v_s\geq w$ for all $s\geq 1$. Assume that $\vert u_s-v_s\vert\rightarrow 0$ in $\tilde{C}_{m}$-capacity. Then the sequence of measures $H_m(u_s)-H_m(v_s)\rightarrow 0$, weakly, as $s\rightarrow+\infty.$
\end{thm}
\begin{cor}\label{00}
Let $u_0\in \mathcal{E}_m$ and $[u_s]\subset \mathcal{E}_m$ be such that $u_0\leq u_s$ for all $s\in \mathbb{N}$. If $u_s$ converes to a $m$-subharmonic function $u$ in $\tilde{C}_{m}$-capacity, then the sequence of measures $H_m(u_s)$ $\rightarrow$ $H_m(u)$ weakly, as $s\rightarrow+\infty.$
\end{cor}
\section{The Comparison Principle in $\mathcal{N}_{m}(H)$}\label{77771}
\subsection{Xing-Type Comparison Principle for $\mathcal{E}_{m}$}
In this section we give the comparison principle for functions in $\mathcal{N}_m(H)$. We shall firstly prove Xing-Type inequaliry for $\mathcal{E}_{m}$ following ideas from \cite{N-Ph}.
\begin{pro}\label{4}
\begin{itemize}
\item[$(a)$] Let $u, v \in \mathcal{F}_{m}$ be such that $u\leq v$ on $\Omega$. Then for $1\leq k\leq m$ and all $r\geq 1,$
  $$
  \frac{1}{k!}\int_{\Omega}(v-u)^{k}dd^{c}\omega_{1}\wedge ...\wedge dd^{c}\omega_{m}\wedge \beta^{n-m} + \int_{\Omega}\tilde{\omega}_{1}(dd^{c}v)^{k}\wedge dd^{c}\omega_{k+1}\wedge...\wedge dd^{c}\omega_{m}\wedge \beta^{n-m}$$
  $$ \leq  \int_{\Omega}\tilde{\omega}_{1}(dd^{c}u)^{k}\wedge dd^{c}\omega_{k+1}\wedge...\wedge dd^{c}\omega_{m}\wedge \beta^{n-m}.
  $$
Where $\omega_j \in \mathcal{SH}_m(\Omega),$ $0\leq \omega_j \leq 1,$ $j=1,\cdots, k$, $\omega_{k+1},\cdots,\omega_{m} \in \mathcal{F}_m$, and $\tilde{\omega}_{1}=r-\omega_{1}.$
 \item[$(b)$] Let $u,\; v \in \mathcal{E}_{m}$ such that $u\leq v$ on $\Omega$ and $u=v$ on $\Omega \backslash K$ for some $K \Subset \Omega.$ Then for $1\leq k\leq m$ and all $r\geq 1,$
   $$
  \frac{1}{k!}\int_{\Omega}(v-u)^{k}dd^{c}\omega_{1}\wedge ...\wedge dd^{c}\omega_{m}\wedge \beta^{n-m} + \int_{\Omega}\tilde{\omega}_{1}(dd^{c}v)^{k}\wedge dd^{c}\omega_{k+1}\wedge...\wedge dd^{c}\omega_{m}\wedge \beta^{n-m}$$
  $$ \leq  \int_{\Omega}\tilde{\omega}_{1}(dd^{c}u)^{k}\wedge dd^{c}\omega_{k+1}\wedge...\wedge dd^{c}\omega_{m}\wedge \beta^{n-m}.
  $$
Where $\omega_j \in \mathcal{SH}_m(\Omega),$ $0\leq \omega_j \leq 1,$ $j=1,\cdots, k$, $\omega_{k+1},\cdots,\omega_{m} \in \mathcal{E}_m$ and $\tilde{\omega}_{1}=r-\omega_{1}.$
\end{itemize}
\end{pro}
For the proof we need the following lemmas.
\begin{lem}\label{5}
Assume that $u, v \in \mathcal{SH}_m(\Omega)\cap L^{\infty}(\Omega)$ such that $u\leq v$ on $\Omega$ and 
$\displaystyle\lim_{z\rightarrow \partial\Omega}\left[u(z)-v(z)\right]=0,$ then
$$\int_{\Omega}(v-u)^{k}dd^{c}\omega\wedge T\leq k\int_{\Omega}(1-\omega)(v-u)^{k-1}dd^{c}u\wedge T.$$
for all $\omega \in \mathcal{SH}_m(\Omega),$ $0\leq\omega \leq 1$ and all closed $m$-positive current $T.$
\end{lem}
\begin{proof}
see [Nh-P, Lemma 3.2]
\end{proof}
\begin{lem}\label{6}
Assume that  $u, v\in\mathcal{SH}_m(\Omega)\cap L^{\infty}(\Omega)$ such that $u\leq v$ on $\Omega$ and 
$\displaystyle\lim_{z\rightarrow \partial\Omega}\left[u(z)-v(z)\right]=0,$ then for $1\leq k\leq m,$ and all $r\geq 1,$
$$
  \frac{1}{k!}\int_{\Omega}(v-u)^{k}dd^{c}\omega_{1}\wedge ...\wedge dd^{c}\omega_{m}\wedge \beta^{n-m} + \int_{\Omega}\tilde{\omega}_{1}(dd^{c}v)^{k}\wedge dd^{c}\omega_{k+1}\wedge...\wedge dd^{c}\omega_{m}\wedge \beta^{n-m}$$
  $$ \leq  \int_{\Omega}\tilde{\omega}_{1}(dd^{c}u)^{k}\wedge dd^{c}\omega_{k+1}\wedge...\wedge dd^{c}\omega_{m}\wedge \beta^{n-m}.
  $$
Where $\omega_j \in \mathcal{SH}_m(\Omega),$ $0\leq \omega_j \leq 1,$ $j=1,\cdots, k$, $\omega_{k+1},\cdots,\omega_{m} \in \mathcal{E}_m$ and $\tilde{\omega}_{1}=r-\omega_{1}.$
\end{lem}
\begin{proof}
Suppose first that $u,\; v \in \mathcal{SH}_{m}(\Omega)\cap L^{\infty}(\Omega),$ $u\leq v$ on $\Omega$, and $u=v$  on $\Omega \backslash K$ with $K \Subset \Omega.$ For simplicity we put $T=dd^{c}\omega_{k+1}\wedge...\wedge dd^{c}\omega_m\wedge \beta ^{n-m}$, then  by lemma \ref{5} we have
\begin{eqnarray*}
\displaystyle\int_{\Omega}(v-u)^{k}dd^{c}\omega_{1}\wedge ...\wedge dd^{c}\omega_{m}\wedge \beta^{n-m}&\leq& k\int_{\Omega}(1-\omega_k)(v-u)^{k-1} dd^{c}u\wedge dd^{c}\omega_{1}\wedge ...\wedge dd^{c}\omega_{k-1}\wedge T \;\;\;\;\;\;\; \;\;\;\;\;\;\;\;\;\;\;\;\;\;\;\;\;\;\;\;\;\;\;\;\;\;\;\;\\
&\leq& k\displaystyle\int_{\Omega}(v-u)^{k-1} dd^{c}u\wedge dd^{c}\omega_{1}\wedge ...\wedge dd^{c}\omega_{k-1}\wedge T\\
 &\leq &  k(k-1)\displaystyle\int_{\Omega}(v-u)^{k-2} (dd^{c}u)^2\wedge dd^{c}u\wedge dd^{c}\omega_{1}\wedge ...\wedge dd^{c}\omega_{k-2}\wedge T\\ 
&\vdots& \\
 &\leq& k!\displaystyle\int_{\Omega}(v-u) (dd^{c}u)^{k-1}\wedge dd^{c}\omega_1\wedge T\\
   &\leq & k!\displaystyle\int_{\Omega}(v-u)\left[\displaystyle\sum_{l=0}^{k-1}(dd^cu)^l\wedge
(dd^cv)^{k-1-l}\right]\wedge dd^{c}\omega_1\wedge T\\
 &=&  k!\displaystyle\int_{\Omega}\tilde{\omega}_{1}dd^c(u-v)\wedge\left[\displaystyle\sum_{l=0}^{k-1}(dd^cu)^l\wedge(dd^cv)^{k-1-l}\right]\wedge T\\
 &=&  k!\displaystyle\int_{\Omega}\tilde{\omega}_{1}(dd^cu)^k\wedge T -k!\int_{\Omega}\tilde{\omega}_{1} (dd^cv)^{k}\wedge T.
\end{eqnarray*}
In the general case, for each $\varepsilon>0,$ we put $v_{\varepsilon}:=\max\{u, v-\varepsilon\}.$ Then, $v_{\varepsilon}\uparrow v\;\hbox{as}\; \varepsilon\; \hbox{tends to}\; 0,\\
 v_{\varepsilon}\geq u\; \hbox{on}\;\Omega$,\; \hbox{and}\; $v_{\varepsilon}=u \;\hbox{on} \;\Omega\backslash K,$
where $K \Subset \Omega,$ then, in one hand we have
$$\frac{1}{k!}\displaystyle\int_{\Omega}\displaystyle(v_{\varepsilon}-u)^{k}dd^{c}\omega_{1}\wedge ...\wedge dd^{c}\omega_{m}\wedge \beta^{n-m}+ \displaystyle\int_{\Omega}\tilde{\omega}_{1}(dd^{c}v_{\varepsilon})^{k}\wedge T\leq \displaystyle\int_{\Omega}\tilde{\omega}_{1}(dd^{c}u)^{k}\wedge T.$$
On the other hand  $0\leq v_{\varepsilon}-u\nearrow v-u$ and by [\cite{Chi1}, Theorem 1.3.10], $(dd^{c}v_{\varepsilon})^k \wedge T$ converges weakly to $(dd^{c}v)^k \wedge T$ as $\varepsilon\downarrow0.$ since $\tilde{\omega}_{1}$ is lower semicontinous, then by letting $\varepsilon$ tends to $0$ we obtain the desired inequality.
\end{proof}
\emph{\textbf{Proof of proposition \ref{4}.}}
\begin{itemize}
\item[$(a)$] Let $[u_j]$, $[v_j] \subset \mathcal{E}^{0}_{m}$  such that $u_j\searrow u$ and $v_j \searrow v,$ as in the definition of $\mathcal{F}_m$. Replacing  $v_j$ by $\max\{u_j, v_j\}$ and using Lemma \ref{6}  we have for all $1\leq j\leq s.$
$$\frac{1}{k!}\displaystyle\int_{\Omega}(v_{j}-u_s)^{k}dd^{c}\omega_{1}\wedge ...\wedge dd^{c}\omega_{m}\wedge \beta^{n-m}+ \displaystyle\int_{\Omega}\tilde{\omega}_{1}(dd^{c}v_{j})^{k}\wedge T
   \leq \displaystyle\int_{\Omega}\tilde{\omega}_{1}(dd^{c}u_s)^{k}\wedge T.$$
where $T=dd^{c}\omega_{1}\wedge...\wedge dd^{c}\omega_m\wedge \beta ^{n-m}$. Let $s\longrightarrow+\infty$ in the above inequality, then by Proposition \ref{29} we get
$$
  \frac{1}{k!}\displaystyle\int_{\Omega}(v_{j}-u)^{k}dd^{c}\omega_{1}\wedge ...\wedge dd^{c}\omega_{m}\wedge \beta^{n-m}+ \displaystyle\int_{\Omega}\tilde{\omega}_{1}(dd^{c}v_{j})^{k}\wedge T
   \leq \displaystyle\int_{\Omega}\tilde{\omega}_{1}(dd^{c}u)^{k}\wedge T,$$
for all  $j\geq 1,$ Finally by letting $j$ tends to $+\infty,$ and again by Proposition \ref{29} we obtain the result.
\item[$(b)$] Let $G, W$ be open sets such that $K \Subset G \Subset W \Subset \Omega,$ by [\cite{Chi1}, Remark 1.7.6] we can find $\widetilde{v}\in\mathcal{F}_m$ with $\widetilde{v}\geq v$ and $\widetilde{v}=v$ on $W$. Set
  $$ \widetilde{u}=\left\{
     \begin{array}{ll}
       u \ \ \ \ \ \  $sur$ \ \  G, \\
        \tilde{v} \ \ \ \ \ \ $sur$ \ \ \Omega\backslash G.
     \end{array}
   \right.$$   
Since  $\widetilde{v}=u=v$ on $W \backslash K,$ then $\widetilde{u} \in \mathcal{SH}^{-}_{m}(\Omega).$ Furthermore $\widetilde{u}\in \mathcal{F}_{m},$ $\widetilde{u}\leq \widetilde{v}$ and $\widetilde{u}=u$ on $W.$ Hence by $a),$ we have
$$
 \frac{1}{k!}\displaystyle\int_{W}(\widetilde{v}-\widetilde{u})^{k}dd^{c}\omega_{1}\wedge ...\wedge dd^{c}\omega_{m}\wedge \beta^{n-m}+ \displaystyle\int_{\Omega}\tilde{\omega}_{1}(dd^{c}\widetilde{v})^{k}\wedge T
  \leq \displaystyle\int_{\Omega}\tilde{\omega}_{1}(dd^{c}\widetilde{u})^{k}\wedge T.$$
However, since $\widetilde{u}=\widetilde{v}$ sur $\Omega\setminus G\supset\Omega\setminus W,$ we get
$$
\frac{1}{k!}\displaystyle\int_{W}(\widetilde{v}-\widetilde{u})^{k}dd^{c}\omega_{1}\wedge ...\wedge dd^{c}\omega_{m}\wedge \beta^{n-m}+\displaystyle\int_{W}\tilde{\omega}_{1}(dd^{c}\widetilde{v})^{k}\wedge T\leq
\displaystyle\int_{W}\tilde{\omega}_{1}(dd^{c}\widetilde{u})^{k}\wedge T.
$$
On the other hand, since $\widetilde{u}=u,\; \widetilde{v}=v$ on $W$ and $u=v$ on $\Omega\backslash K\supset\Omega\setminus W,$  we obtain
$$ \frac{1}{k!}\displaystyle\int_{\Omega}(v-u)^{k}dd^{c}\omega_{1}\wedge ...\wedge dd^{c}\omega_{m}\wedge \beta^{n-m}+ \displaystyle\int_{\Omega}\tilde{\omega}_{1}(dd^{c}v)^{k}\wedge T\leq \displaystyle\int_{\Omega}\tilde{\omega}_{1}(dd^{c}u)^{k}\wedge T.$$\qed 
\end{itemize}
For similar result we can see \cite{Dh-Elkh}.
\begin{thm}\label{11}
Let $u, v \in \mathcal{E}_m$ and $1\leq k\leq m$ are such that
$\displaystyle\liminf_{z\rightarrow \partial\Omega}\left[u(z)-v(z)\right]\geq0.$ Then  we have
$$
\frac{1}{k!}\displaystyle\int_{\{u<v\}}(v-u)^{k}dd^{c}\omega_{1}\wedge ...\wedge dd^{c}\omega_{m}\wedge \beta^{n-m}
 + \displaystyle\int_{\{u<v\}}\tilde{\omega}_{1}(dd^{c}v)^{k}\wedge dd^{c}\omega_{k+1}\wedge ...\wedge dd^{c}\omega_{m}\wedge \beta^{n-m}
 $$
 $$
  \leq  \displaystyle\int_{\{u<v\}\cup\{u=v=-\infty\}}\tilde{\omega}_{1}(dd^{c}u)^{k}\wedge dd^{c}\omega_{k+1}\wedge ...\wedge dd^{c}\omega_{m}\wedge \beta^{n-m},
 $$
for all $\omega_j \in \mathcal{SH}_m(\Omega),$ $0\leq \omega_j \leq 1,$ $j=1,\cdots, k$, $\omega_{k+1},\cdots,\omega_{m} \in \mathcal{E}_m$ and all $r\geq 1$ such that $\tilde{\omega}_{1}=r-\omega_{1}.$
\end{thm}
\begin{proof}
For each $\varepsilon>0$ we put $\widetilde{v}=\max(u,v-\varepsilon).$
By applying $b)$ in Proposition \ref{4} for $u$ and $\widetilde{v},$ we obtain
$$
\frac{1}{k!}\displaystyle\int_{\Omega}(\widetilde{v}-u)^{k}dd^{c}\omega_{1}\wedge ...\wedge dd^{c}\omega_{m}\wedge \beta^{n-m}
 + \displaystyle\int_{\Omega}\tilde{\omega}_{1}(dd^{c}\widetilde{v})^{k}\wedge T
  \leq  \displaystyle\int_{\Omega}\tilde{\omega}_{1}(dd^{c}u)^{k}\wedge T.
$$
Where T$=dd^{c}\omega_{k+1}\wedge ...\wedge dd^{c}\omega_{m}\wedge \beta^{n-m}$. Since $\{u<\tilde{v}\}=\{u<v-\varepsilon\},$ by lemma \ref{8}, we have
\begin{flushleft}
  $\frac{1}{k!}\displaystyle\int_{\{u<\widetilde{v}\}}(\widetilde{v}-u)^{k}dd^{c}\omega_{1}\wedge ...\wedge dd^{c}\omega_{m}\wedge \beta^{n-m}
+ \displaystyle\int_{\{u<\widetilde{v}\}}\tilde{\omega}_{1}(dd^{c}\widetilde{v})^{k}\wedge T$
 
  $\;\;\;\;\;\;\;\;\;\;\;\;\leq\frac{1}{k!}\displaystyle\int_{\{u<v-\varepsilon\}}(\widetilde{v}-u)^{k}dd^{c}\omega_{1}\wedge ...\wedge dd^{c}\omega_{m}\wedge \beta^{n-m}
+ \displaystyle\int_{\{u\leq v-\varepsilon\}}\tilde{\omega}_{1}(dd^{c}\widetilde{v})^{k}\wedge T$

 $\;\;\;\;\;\;\;\;\;\;\;\;\leq \frac{1}{k!}\displaystyle\int_{\Omega}(\widetilde{v}-u)^{k}dd^{c}\omega_{1}\wedge ...\wedge dd^{c}\omega_{m}\wedge \beta^{n-m}
+ \displaystyle\int_{\Omega}\tilde{\omega}_{1}(dd^{c}\widetilde{v})^{k}\wedge T-\displaystyle\int_{\{u>v-\varepsilon\}}\tilde{\omega}_{1}(dd^{c}\widetilde{v})^{k}\wedge T $

  $\;\;\;\;\;\;\;\;\;\;\;\;\leq \displaystyle\int_{\Omega}\tilde{\omega}_{1}(dd^{c}u)^{k}\wedge T-\displaystyle\int_{\{u>v-\varepsilon\}}\tilde{\omega}_{1}(dd^{c}\widetilde{v})^{k}\wedge T$

  $\;\;\;\;\;\;\;\;\;\;\;\;= \displaystyle\int_{\{u\leq v-\varepsilon\}}\tilde{\omega}_{1}(dd^{c}u)^{k}\wedge T$
  
  $\;\;\;\;\;\;\;\;\;\;\;\;\leq \displaystyle\int_{\{u<v\}\cup\{u=v=-\infty\}}\tilde{\omega}_{1}(dd^{c}u)^{k}\wedge T.$
\end{flushleft}
Letting $\varepsilon\searrow0$ we obtain the desired inequality.
\end{proof}
We now prove a Xing-type comparison principle for the class $\mathcal{N}_{m}(H)$.
\begin{thm}\label{25}
Let $H \in \mathcal{E}_{m},$ if $u \in \mathcal{N}_{m}(H)$ and $v \in \mathcal{E}_{m}$ such that  $v\leq H$ on $\Omega.$ Then For all sequence $[\omega_{k}]_{1\leq k\leq m} \subset \mathcal{SH}_{m}(\Omega)\cap L^{\infty}(\Omega)$ with $-1\leq\omega_{k}\leq0,$  we have  
$$
  \frac{1}{m!}\displaystyle\int_{\{u<v\}}(v-u)^{m}dd^{c}\omega_{1}\wedge ...\wedge dd^{c}\omega_{m}\wedge \beta^{n-m}
+ \displaystyle\int_{\{u<v\}}(-\omega_{1})H_m(v)\;\;\;\;$$
$$
  \leq  \displaystyle\int_{\{u<v\}}(-\omega_{1})H_m(u)+\int_{\{u=v=-\infty\}}(-\omega_{1})H_m(u).
$$
\end{thm}
\begin{proof}
Let $\omega_1\in \mathcal{SH}_{m}(\Omega)\cap L^{\infty}(\Omega)$ such that $-1\leq\omega_1\leq 0$ and $u \in \mathcal{N}_{m}(H)$, then there exists a function $\phi \in \mathcal{N}_{m}$ such that
                                             $H\geq u\geq \phi+H.$
Let $[\Omega_{j}]$ be the fondamental sequence of  $\Omega$ and $\phi^{j}$ defined as in Definition \ref{120}. Since $v\leq H$ this implies that for $\varepsilon > 0$ we have
                 $ u\geq \phi+H=\phi^{j}+ H\geq v+\phi^{j}-\varepsilon$ on $\Omega\setminus\Omega_j.$
Then by applying  Theorem \ref{11} for $u$, $v+\phi^{j}-\varepsilon$, $r=1$
and $\tilde{\omega}_1=\omega_1+1$ (in this case $0\leq\tilde{\omega}_1\leq1$), we get
$$
 \frac{1}{m!}\displaystyle\int_{\{u<v-\varepsilon+\phi^j\}}(v-\varepsilon+\phi^j-u)^{m}dd^{c}\omega_{1}\wedge ...\wedge dd^{c}\omega_{m}\wedge \beta^{n-m}
+\displaystyle\int_{\{u<v-\varepsilon+\phi^j\}}(-\omega_{1})H_m(v+\phi^j)$$
$$
\leq \displaystyle\int_{\{u\leq v-\varepsilon\}}(-\omega_{1})H_m(u).
$$
On the other hand, $\left[\chi_{\{u<v-\varepsilon+ \phi^{j}\}}\right]^{\infty}_{j=1}$ \; and \; $\left[\chi_{\{u<v-\varepsilon+ \phi^{j}\}}(v-\varepsilon+ \phi^{j}-u)^{m}\right]^{\infty}_{j=1}$
are two increasing sequences of functions that converges $q.e.$ on $\Omega$ to $\chi_{\{u<v-\varepsilon\}}$ and $\chi_{\{u<v-\varepsilon\}}(v-\varepsilon-u)^{m}$ respectively, as $j\rightarrow+\infty.$  Theorem \ref{51} implies that
$dd^{c}\omega_{1}\wedge ...\wedge dd^{c}\omega_{m}\wedge \beta^{n-m}
\ll C_m$ and $\chi_{\{v>-\infty\}}H_m(v)\ll C_m$. Therefore we get that $\left[\chi_{\{u<v-\varepsilon+ \varphi^{j}\}}\right]^{\infty}_{j=1}$ converges to $\chi_{\{u<v-\varepsilon\}}$ a.e. w.r.t. $\chi_{\{v>-\infty\}}H_m(v)$ and that
$\left[\chi_{\{u<v-\varepsilon+ \phi^{j}\}}(v-\varepsilon+ \phi^{j}-u)^m\right]^{\infty}_{j=1}$ converges to $\chi_{\{u<v-\varepsilon\}}(v-\varepsilon-u)^m$ a.e. w.r.t. $dd^{c}\omega_{1}\wedge ...\wedge dd^{c}\omega_{m}\wedge \beta^{n-m}
.$ Therefore, by the monotone convergence theorem we obtain
$$
  \frac{1}{m!}\displaystyle\int_{\{u<v-\varepsilon\}}(v-\varepsilon-u)^{m}dd^{c}\omega_{1}\wedge ...\wedge dd^{c}\omega_{m}\wedge \beta^{n-m}
+\displaystyle\int_{\{u<v-\varepsilon\}}(-\omega_{1})H_m(v) \leq\displaystyle\int_{\{u\leq v-\varepsilon\}}(-\omega_{1})H_m(u).
$$
The desired inequality is obtained by letting $\varepsilon \rightarrow 0^+$.
\end{proof}
\subsection{The Comparison and Identity Principles for the class $\mathcal{N}_{m}(H),$ $H \in \mathcal{E}_{m}$ }
We give now one of the most important result which will play a crucial role later in this paper: The comparison principle.
\begin{cor}\label{14}(The Comparison Principle).
Let $u, v, H \in \mathcal{E}_{m}$ be such that $H_m(u)$ vanishes on all $m$-polar sets in $\Omega$ and $H_m(u)\leq H_m(v)$. Consider the following two conditions:
\begin{description}
\item[(1)] $\displaystyle\liminf_{z\rightarrow\zeta }\left[u(z)-v(z)\right]\geq0\;\;\; \hbox{for\; every}\;\; \zeta \in \partial\Omega,$
\item[(2)]  $u \in \mathcal{N}_m(H), v\leq H.$
\end{description}
If one of the above conditions is satisfied, then $u\geq v$ on $\Omega.$
\end{cor}
\begin{proof}
Suppose that $u, v, H \in \mathcal{E}_{m}$, such that $H_m(u)$ vanishes on all  $m$-polar sets in  $\Omega$ and $H_m(u)\leq H_m(v).$
\begin{itemize}
  \item[$(1)$] Let $\varepsilon>0$. Suppose that $\displaystyle\liminf_{z\rightarrow\zeta}\left[u(z)-v(z)\right]\geq0,$ for all $\zeta \in \partial\Omega$. Then by Theorem \ref{11} applied for $\omega=-\tilde{\omega}_j$ we have
\begin{eqnarray*}
    \frac{\varepsilon^m}{m!}C_m(\{u+2\varepsilon<v\})= \sup\displaystyle\left\{\frac{\varepsilon^m}{m!}\int_{\{u+2\varepsilon<v\}}H_m(\omega):\; \omega \in \mathcal{SH}_m(\Omega),\; -1\leq\omega\leq0\right\}\;\;\;\;\;\;\;\;\;\;\;\;\;\;\;\;\;\;\;\;\;\;\;\; &&\\
     \leq\sup\displaystyle\left\{\frac{1}{m!}\int_{\{u+\varepsilon<v\}}(v-\varepsilon-u)^mH_m(\omega):\; \omega \in \mathcal{SH}_m(\Omega),\; -1\leq\omega\leq0\right\}\;\;\;\;\;\;\;\;\;\;\;&& \\
     \leq\displaystyle\int_{\{u+\varepsilon<v\}}(-\omega)\left[H_m(u)-H_m(v)\right]\leq 0.\;\;\;\;\;\;\;\;\;\;\;\;\;\;\;\;\;\;\;\;\;\;\;\;\;\;\;\;\;\;\;\;\;\;\;\;\;\;\;\;\;\;\;\;\;\;\;\;\;\;\;\;\;\;\;(4)&&
\end{eqnarray*}
Thus, $u+2\varepsilon\geq v$. Let $\varepsilon$ tends to $0^+$, then we obtain that $u\geq v$ on $\Omega$.
  \item[$(2)$] Suppose now that $u \in \mathcal{N}_m(H)$ and $v\leq H.$ Then there exists $\varphi \in \mathcal{N}_m$ such that $H+\varphi\leq u\leq H$. Let $\varphi^j$ be defined as in  Definition \ref{120} and let $\varepsilon>0.$ By Theorem \ref{25} and using the same argument as in (4) for $\{u+2\varepsilon-\varphi^j<v\}$ we have $u+2\varepsilon\geq v+\varphi^j$. By letting $\varepsilon \rightarrow 0^+$, we obtain the inequality.
\end{itemize}
\end{proof}
\begin{lem}\label{18}
Let $\varphi \in \mathcal{SH}^{-}_m(\Omega)$ and $u, v \in\mathcal{N}_m(H)$ are such that $u\leq v$ and $T$ is a closed $m$-positive current of type $T=dd^c\omega_2\wedge...\wedge dd^c\omega_m\wedge\beta^{n-m}$, where $\omega_j \in \mathcal{E}_m$, $\forall j$. If $\displaystyle\int_{\Omega}(-\varphi)dd^cu\wedge T<+\infty.$ Then the following inequality holds
  \begin{equation}\label{13}
    \displaystyle\int_{\Omega}(-\varphi)dd^cv\wedge T\leq\displaystyle\int_{\Omega}(-\varphi)dd^cu\wedge T.
  \end{equation}
\end{lem}
\begin{proof}
Let $[\Omega_{j}]$ be a fondamental sequence of $\Omega$ and $u\in \mathcal{N}_m(H),$ then there exists a function  $\psi \in \mathcal{N}_m$ such that $H\geq u\geq H+\psi.$ We set $v_j=\max(u, \psi^j+v),$ then $v_j \in \mathcal{E}_m, v_j=u$ on $\Omega\setminus\Omega_j,$ $u\leq v_j$ and $[v_j]$ is an increasing sequence that converges to $v$  q.e. on  $\Omega$, as $j\rightarrow +\infty.$
On the other hand, since $\varphi \in \mathcal{SH}_m(\Omega)$ then by  \cite[Theorem 3.1]{Chi2} there exists $[\varphi_k]\subset \mathcal{E}^0_m\cap C(\overline{\Omega})$, that converges pointwise to $\varphi$ as $k\rightarrow +\infty.$ Hence, by the stockes Theorem we obtain for $r\geq j$
$$
  \int_{\Omega_r}(-\varphi_k)dd^cu\wedge T- \int_{\Omega_r}(-\varphi_k)dd^cv_j\wedge T  = \int_{\Omega_r}(-\varphi_k)dd^c(u-v_j)\wedge T = \int_{\Omega_r}(v_j-u)dd^c\varphi_k\wedge T\geq 0.
$$
By\;letting\;  $r\rightarrow+\infty\;$ we\;get
\begin{equation}\label{15}
 \displaystyle\int_{\Omega}(-\varphi_k)dd^cu\wedge T\geq \displaystyle\int_{\Omega}(-\varphi_k)dd^cv_j\wedge T.\;\;\;\;\;\;\;\;\;\;\;\;
 \end{equation}
Since $v_j$ converges q.e to $v$, then $v_j$ converges to $v$ in $C_m$-capacity and since $\varphi_k$ is bounded, then it follows from Corollary \ref{00} that 
          $(-\varphi_k)dd^cv_j\wedge T$ converges weakly to $(-\varphi_k)dd^cv\wedge T\;\; \hbox{as}\;\; j\rightarrow +\infty.$ Thus
\begin{equation}\label{16}
\lim_{j\rightarrow+\infty}\displaystyle\int_{\Omega}(-\varphi_k)dd^cv_j\wedge T\geq \int_{\Omega}(-\varphi_k)dd^cv\wedge T.
 \end{equation}
Inequalities  (\ref{15}) and (\ref{16}) imply that (\ref{13}) holds for $\varphi_k$. By the  monotone convergence theorem we completes the proof, when we let $k\rightarrow+\infty.$
\end{proof}
\begin{pro}\label{19}
Let $H \in \mathcal{E}_{m}$ and $\varphi\in\mathcal{SH}^{-}_m(\Omega)$. If $[u_j]$, $u_j \in \mathcal{N}_m(H)$, is a decreasing sequence that  converges pointwise on $\Omega$ to a function $u \in \mathcal{N}_m(H)$ as $j\rightarrow +\infty$, then
 \begin{equation}\label{17}
      \displaystyle\lim_{j\rightarrow+\infty}\int_{\Omega}(-\varphi)H_m(u_j)= \int_{\Omega}(-\varphi)H_m(u).
 \end{equation}
\end{pro}
\begin{proof}
Let $\varphi \in \mathcal{SH}^{-}_m(\Omega)$ and  $u_j, u \in \mathcal{N}_m(H)$ such that $u\leq u_j.$ If $\displaystyle\int_{\Omega}(-\varphi)H_m(u)=+\infty,$ then    \;\;\;\;\;\;\;\;\;\;\;\;\;\;\;\;\;\; $\lim_{j\rightarrow+\infty}\displaystyle\int_{\Omega}(-\varphi)H_m(u_j)=+\infty$ and (\ref{17}) holds. Therefore we can assume that $\displaystyle\int_{\Omega}(-\varphi)H_m(u)<+\infty.$ Lemma \ref{18} implies that the sequence $\displaystyle{[\int_{\Omega}(-\varphi)H_m(u_j)]_j}$ is an increasing sequence  that is bounded from above by $\displaystyle\int_{\Omega}(-\varphi)H_m(u).$ So by  the same argument as in the proof of Proposition \ref{29}, the sequence
 $[(-\varphi)H_m(u_j)]_j$ converges weakly to $(-\varphi)H_m(u)$, and the desired limit of the total masses is valid.
\end{proof}
\begin{lem}\label{80}
Let $H \in \mathcal{E}_m$ and $u, v \in \mathcal{N}_m(H)$ such that $u\leq v.$ Then for all
$\omega_j \in \mathcal{SH}_m(\Omega)\cap L^{\infty}(\Omega),$ $-1\leq \omega_j\leq0,$ $j=1,...,m$, $\displaystyle\int_{\Omega}(-\omega_1)H_m(u)<+\infty$, we have that the following inequality holds:
\begin{equation}\label{20}
  \frac{1}{m!}\displaystyle\int_{\Omega}(v-u)^{m}dd^{c}\omega_{1}\wedge ...\wedge dd^{c}\omega_{m}\wedge \beta^{n-m}
 + \displaystyle\int_{\Omega}(-\omega_{1})H_m(v)\leq  \displaystyle\int_{\Omega}(-\omega_{1})H_m(u).
\end{equation}
\end{lem}
\begin{proof}
\emph{\textbf{Step 1:}} Suppose first that $u, v \in \mathcal{E}^0_m(H).$ Then by definition there exists a  function $\varphi \in \mathcal{E}^0_m$ such that
$H\geq u\geq  H+\varphi.$ For each $\varepsilon>0$ small  enough we can choose $K\Subset\Omega$ such that $\varphi\geq -\varepsilon$ on $\Omega\setminus K$. Hence,
                $$u\geq\varphi+H\geq-\varepsilon+H\geq-\varepsilon+v \;\;\;\hbox{on}\;\;\; \Omega\setminus K,$$
Put $\widehat{u}=\max(u,v-\varepsilon)$, then $\widehat{u}=u$ on $\Omega\setminus K.$ BAy $(b)$ in Proposition \ref{4} we have
\begin{equation*}
\frac{1}{m!}\displaystyle\int_{\Omega}(\widehat{u}-u)^{m}dd^{c}\omega_{1}\wedge ...\wedge dd^{c}\omega_{m}\wedge \beta^{n-m}
+ \displaystyle\int_{\Omega}(-\omega_{1})H_m(\widehat{u})\leq \displaystyle\int_{\Omega}(-\omega_{1})H_m(u).
\end{equation*}
By letting $\varepsilon\rightarrow 0^+$ (\ref{20}) holds.
\par \emph{\textbf{Step 2:}} Let now $u, v \in \mathcal{N}_m(H),$ then by Proposition \ref{0} there exist two decreasing sequences $[u_j]$, $[v_k]$ such that $u_k, v_j \in\mathcal{E}^0_m(H)$ that converge pointwise to $u$ and $v$ respectively, so using the first part we get
 \begin{equation*}
\frac{1}{m!}\int_{\Omega}(v_k-u_j)^{m}dd^{c}\omega_{1}\wedge ...\wedge dd^{c}\omega_{m}\wedge \beta^{n-m}
 + \int_{\Omega}(-\omega_{1})H_m(v_k)\leq \int_{\Omega}(-\omega_{1})H_m(u_j).
\end{equation*}
Finally by Proposition \ref{19} we get the desired inequality.
\end{proof}
An immediate consequence of Lemma \ref{80} is the following identity principle for the class $\mathcal{N}_{m}(H),$ Theorem \ref{31} will play a crucial role in this paper.
\begin{thm}\label{31}(The Identity Principle).
Let $H \in \mathcal{E}_m.$ If $u,v\in\mathcal{N}_m(H)$ such that $u\leq v$, $H_m(u)=H_m(v)$ and $\displaystyle\int_{\Omega}(-\omega)H_m(u)<+\infty$ with $\omega \in\mathcal{E}_m$, then $u=v$ on $\Omega$.
\end{thm}
\section{The Dirichlet problem in $\mathcal{N}_m(H),$ $H\in\mathcal{MSH}_{m}(\Omega)$}\label{77772}
In this section, we formulate one of our main results, Theorem \ref{52}. We shall first study the Dirichlet problem with continuous boudary data: 
\begin{equation}\label{1003}
\left\{
     \begin{array}{ll}
        u\in\mathcal{SH}_m(\Omega)\cap L^{\infty}_{loc}(\Omega),\\
        H_m(u)=d\mu,\\
        \displaystyle\limsup_{z\rightarrow\xi }u(z)=H(\xi), \;\;\;\; \forall\xi\in \partial\Omega.
     \end{array}
   \right.
\end{equation}
Where $\Omega$ is a bounded $m$-pseudoconvex domain, $H\in\mathcal{MSH}_{m}(\Omega)\cap C(\overline{\Omega})$ and
$\mu$ is a positive measure on $\Omega$. Suppose that the class of the class of $m$-subharmonic functions 
$$
\mathcal{B}_m(\mu,H):=\{v\in\mathcal{SH}_m(\Omega)\cap L^{\infty}_{loc}(\Omega): H_m(v)\geq\mu, \limsup_{z\rightarrow\xi}v(z)\leq H(\xi),\forall \xi\in\partial\Omega\}
$$
is non-empty. Then 
$$U_m(\mu, H)=\sup\{v: v\in\mathcal{B}_m(\mu,H)\}\in \mathcal{B}_m(\mu,H).$$
Indeed, it follows from Choquet's lemma that there exists $\varphi_j\in\mathcal{B}_m(\mu,H)$ for $j\in\mathbb{N}$ such that $(\sup_{j\in\mathbb{N}}\varphi_j)^{*}=U^{*}_m(\mu,H)$, where $\theta^{*}$ denotes the smallest upper semicontinuous majorant of $\theta.$ Thus, $U^{*}_m(\mu,H)\in\mathcal{SH}_m(\Omega),$ and since $\varphi_j\leq H$, (because of the $m$-maximality of $H$), then it follows that $\limsup_{z\rightarrow\xi}U^{*}_m(\mu,H)\leq H(\xi)$, for all $\xi\in \partial\Omega$. On the other hand, it follows from \cite{Chi1}, Theorem 1.3.16] that $H_m(\max_{1\leq j\leq k}\varphi_{j})\geq \mu,$ and since $\max_{1\leq j\leq k}(\varphi_{j})\nearrow U^{*}_m(\mu,H)$ a.e as $k\rightarrow+\infty$, then by [\cite{Chi1}, Theorem 1.3.10] we have $H_m(U^{*}_m(\mu,H))\geq \mu,$ so $U^{*}_m(\mu,H)\in\mathcal{B}_m(\mu,H)$ and the proof is complete.
\subsection{The Dirichlet problem with smooth boundary data}
\par We assume in this part that $\Omega$ is smoothly bounded strictly $m$-pseudoconvex and that $H\in\mathcal{MSH}_{m}(\Omega)\cap C^{\infty}(\partial\Omega)$. Note that in this case, It follows from \cite{Li} That $U_m(dV,H),U_m(dV,-H)\in\mathcal{SH}_{m}(\Omega)\cap C^{\infty}(\overline{\Omega})$, so, by \cite[Theorem 3.22]{Chi2} 
$$\int_{\Omega}H_m(H+U_m(0,-H))\leq \int_{\Omega}H_m(U_m(dV,H)+U_m(dV,-H))<+\infty,$$
Let denote by $H^{-}:=U_m(0,-H)$. We have $H+H^{-}\in\mathcal{E}_m^{0}$, and if $\varphi\in\mathcal{E}_m^{0}$, such that $\mu\leq H_m(\varphi)$, then 
$$\int_{\Omega}H_m(U_m(\mu,H))\leq \int_{\Omega}H_m(\varphi+H)\leq \int_{\Omega}H_m(\varphi+H+H^{-})<+\infty,$$
since $\mathcal{E}_m^{0}$ is a convex cone. Thus, if $\mu\leq H_m(\varphi)$, where $\varphi\in\mathcal{E}_m^{0}$,  then we have $H\geq U_m(\mu,H)\geq \varphi+H$, namely $U_m(\mu,H)\in \mathcal{E}_m^{0}(H).$ Furtheremore, by \cite[Theorem 3.9]{Chi2} $U_m(\mu,H)+H^{-}\in \mathcal{E}_m^{0}$.
\begin{lem}\label{55}
Let $\mu$ be a positive and compactly supported measure in $\Omega$ and let $A>0$ a conctant such that
\begin{equation}\label{1002}
\int(-\varphi)^pd\mu\leq A\left(\int(-\varphi)^pH_m(\varphi)\right)^{\frac{p}{p+m}},\;\; \forall\varphi\in\mathcal{E}_m^0
\end{equation}
with $p>\frac{m}{m-1}$, and let $u_j \in \mathcal{E}^0_m\cap C(\overline{\Omega})$. If $u_j$ converges to a function $u\in \mathcal{SH}_m(\Omega)$, as $j$  tends to $+\infty$, a.e. $dV$ and if $\displaystyle\sup_s\int_\Omega H_m(u_s)<+\infty,$ then $\displaystyle\lim_{s\rightarrow+\infty}\int_{\Omega}u_sd\mu=\displaystyle\int_{\Omega}ud\mu.$
\end{lem}
\begin{proof}
Since $\displaystyle\limsup_{s\rightarrow+\infty}\int u_sd\mu\leq \displaystyle\int ud\mu$, so it is enough to prove that  $\displaystyle\liminf_{s\rightarrow+\infty}\int u_sd\mu\geq \displaystyle\int ud\mu.$
For each $N\in \mathbb{N}$ set $\Gamma^s_N=\{u_s<-N\}\cap\hbox{supp}\mu$ and denote by $U_{m,\Gamma_N^s}$ the relative $m$-extremal function of this set, namely 
$$U_{m,\Gamma_N^s}=\sup\{\varphi\in\mathcal{SH}_m(\Omega)\;/\;\varphi\leq0,\;\hbox{and}\; \varphi\leq -1 \;\hbox{in}\;\Gamma_N^s \}$$
 By (\ref{1002}) we have
$$
\int_{\Gamma^s_N}d\mu = A\left(\int_{\Gamma^s_N}H_m(U_{m,\Gamma_N^s})\right)^{\frac{p}{p+m}}.
$$
and thanks  to \cite[Theorem 5.2]{Chi2} we have
$$
\int_{\Omega}H_m(U_{m,\Gamma_N^s})= \int_{\overline{\Gamma}_N^s}H_m(U_{m,\Gamma_N^s})
   \leq\int_{\{\frac{2u_s}{N}<U_{m,\Gamma_N^s}\}}H_m(U_{m,\Gamma_N^s})
   \leq \left(\frac{2}{N}\right)^m\int H_m(u_s)\leq  \left(\frac{2}{N}\right)^m\alpha,
$$
where $\alpha=\sup\int(dd^cu_s)^m\wedge \beta^{n-m}.$ Hence
$$\int_{\Gamma^s_{2^N}}d\mu\leq A\left(2\alpha\right)^{\frac{p}{m+p}}\frac{1}{N^{mp/(m+p)}}.
$$
Since $p>\frac{m}{m-1},$ then $\delta:=mp/(m+p)>1,$ Therfore
$$\int_{\Gamma^s_{2^N}}(-u_s)d\mu=\sum_{j=N}^{\infty}\int_{\{-2^{k+1}\leq u_j\leq -2^{k}\}}\leq A\left(2\alpha\right)^{\frac{p}{m+p}}\sum^{\infty}_{j=N}\frac{2^{j+1}}{2^{k\delta}}\rightarrow0,\;\; N\rightarrow\infty.$$

Hence, $\displaystyle\lim_{N\rightarrow+\infty}\;\sup_j\int_{\Gamma^s_{2^N}}(-u_s)d\mu=0$ and 
\begin{equation}\label{1011}
\int_{\Omega}(-u_s)d\mu = \int_{\{u_s\geq -2^N\}}(-u_s)d\mu+\int_{\Gamma^s_{2^{N}}}(-u_s)d\mu\leq 2^N\int_{\Omega} d\mu+A\left(2\alpha\right)^{\frac{p}{m+p}}\sum^{\infty}_{j=N}\frac{2^{j+1}}{2^{k\delta}}<+\infty.
\end{equation}
In particular, $\displaystyle\sup_s\int_{\Omega}-u_sd\mu<+\infty$, so it is enough to prove that
$\displaystyle\int(-\max\{u_s,-N\})d\mu \rightarrow \displaystyle\int-\max\{u,-N\}d\mu,$
or we can assume that $[u_j]$ is uniformly bounded. Then the sequence is also bounded in $L^{2}(d\mu)$ and passing to a subsequence one can find $v\in L^{2}(d\mu)$ and a subsequence $u_{s_{t}}$ so that $(1\diagup M)\sum^{M}_{t=1}u_{s_{t}}\rightarrow v$ in $L^{2}(d\mu)$. Furthermore there exists a subsequence $M_q$ such that $f_q=(1/M_q)\sum^{M_q}_{t=1}u_{s_{t}}\rightarrow v$ a.e $d\mu$. Since $f_q$ converges in $L^{2}(dV),$  then $\displaystyle(\sup_{r\geq q}f_r)^\ast\searrow u$ everywhere, and
$\displaystyle\int\displaystyle(\sup_{r\geq q}f_r)^\ast d\mu=\displaystyle\int\sup_{r\geq q}f_rd\mu\longrightarrow \int vd\mu \;\;as\; q\longrightarrow+\infty.$
Now since $\mu$ puts no mass on an $m$-polar sets and $f_r\longrightarrow v$ a.e. $d\mu$, then  $\displaystyle\int ud\mu=\displaystyle\int vd\mu=\displaystyle\lim\int u_{s_{t}}d\mu.$
\end{proof}
\begin{lem}\label{1006}
Suppose that $\mu$ is a positive measure with compact support in $\Omega$ such that $\mu$ satisfies (\ref{1002}) for some $p>m/(m-1)$. Assume that $[u_j]\subset \mathcal{E}^{0}_{m}(H)\cap C(\overline{\Omega})$, with $H\in\mathcal{MSH}_m(\Omega)\cap C^{\infty}(\partial\Omega)$, and $u_j\rightarrow u \in\mathcal{SH}_m(\Omega)$ $a.e.\; dV$. $j\rightarrow+\infty$ and that $\displaystyle\sup_j\int_{\Omega}H_m(u_j)<+\infty$. Then $\displaystyle\lim_{j\rightarrow+\infty}\int u_jd\mu=\int ud\mu.$
\end{lem}
\begin{proof}
Since $u_j\in\mathcal{E}^{0}_{m}(H)$, then as shown above we have that $u_j+H^{-}\in \mathcal{E}_m^{0}$, then $H_m(u_j+H^{-})$ satisfies (\ref{1002}), so it is the same for $H_m(u_j)$. Then, it follows from \cite[Theorem 6.2]{Chi2} that there is $\varphi_j:=U_m(H_m(u_j),0)$. Hence, using the comparison principle (Corollary \ref{14}) for $u_j$ and $\varphi_j+H$, we get $u_j\geq \varphi_j+H,$ so $u_j+H^{-}\geq \varphi_j+H+H^{-}.$ For simplicity put $\upsilon_j=:\varphi_j+H+H^-$ Let $1 <\delta <2$. Denote by $A:=\left\{\varphi_j=\delta(H+H^{-})\right\}\subset \{\dfrac{\delta}{\delta-1}\varphi_j<\upsilon_j\}$, $B:=\{\varphi_j<\delta(H+H^{-})\}=\{\dfrac{\delta+1}{\delta}\varphi_j<\upsilon_j\}$, $C:=\{\varphi_j>\delta(H+H^{-})\}=\{(\delta+1)(H+H^{-})<\upsilon_j\}$, then by \cite[(c) in Corollary 1.15 $\&$ Theorem 1.14]{Chi2} we have 
\begin{eqnarray*}
\int_{\Omega} H_m(u_j+H^-)&\leq &\int_{\Omega} H_m(\upsilon_j)=\int_{A}H_m(\upsilon_j)+\int_{B}H_m(\upsilon_j)+\int_{C}H_m(\upsilon_j)\\
    &\leq & (\dfrac{\delta}{\delta-1})^{m}\int_{\Omega}H_m(\varphi_j)+(\dfrac{\delta+1}{\delta})^{m}\int_{\Omega}H_m(\varphi_j)+(\delta+1)^{m}\int_{\Omega}H_m(H+H^{-}),\\
    &\leq &C_{\delta,m}\left[\int_{\Omega}H_m(u_j)+\int_{\Omega}H_m(H+H^{-})\right]<+\infty,
\end{eqnarray*}
where $C_{\delta,m}$ is a constant depending on $\delta$ and $m$. Thus, $\displaystyle\sup_j\int_{\Omega} H_m(u_j+H^-)<+\infty.$ Finally, by Proposition \ref{55} we obtain 
$$\lim_{j\rightarrow+\infty}\int_{\Omega} (u_j+H^{-})d\mu=\int_{\Omega} (u+H^{-})d\mu,$$
which proves the lemma.
\end{proof}
\begin{lem}\label{1007}
Suppose that $H\in\mathcal{MSH}_m(\Omega)\cap C^{\infty}(\partial\Omega)$ and that $u_j\in \mathcal{E}^{0}_{m}(H)\cap C(\overline{\Omega})$. If $u_j\rightarrow u \in\mathcal{SH}_m(\Omega)$ $a.e.\; dV$ as $j\rightarrow+\infty$, $\sup_j\displaystyle\int_{\Omega}(-u_j)H_m(u_j)<+\infty$, and that  $\displaystyle\int_{\Omega}\vert u-u_s\vert H_m(u_j)\rightarrow0$, as $j\rightarrow+\infty.$ Then $H_m(u_j)$ converges weakly to $H_m(u).$
\end{lem}
\begin{proof}
We can assme that $\displaystyle\int_{\Omega}\vert u-u_s\vert H_m(u_j)<1/j^{2}.$ Put $G_j=\max\{u_j+1/j,u\}-1/j$, then for 
$\chi\in C^{\infty}_{0}(\Omega),$
\begin{eqnarray*}
\left\vert\int_{\Omega} \chi H_m(u)-\int_{\Omega} \chi H_m(u_j)\right\vert &=& \left\vert\int_{\Omega} \chi[H_m(u)-H_m(G_j)+H_m(G_j)-H_m(u_j)]\right\vert\\
    &\leq & \left\vert\int_{\Omega} \chi[H_m(u)-H_m(G_j)]\right\vert+\left\vert\int_{\{u_j+1/j\leq u\}}\chi\left[ H_m(G_j)-H_m(u_j)\right]\right\vert\\
&\leq & \sup\vert\chi\vert\left(\underbrace{\int_{\Omega}\left\vert H_m(u)-H_m(G_j)\right\vert}_{\tau_j}+2\int_{\{u_j+1/j\leq u\}} H_m(u_j)\right).
\end{eqnarray*} 
But, $$\int_{\{u_j+1/j\leq u\}} H_m(u_j)\leq j\int_{\Omega} \vert u-u_j\vert H_m(u_j)\leq 1/j,$$ so it remains to estimate $\tau_j.$ Let $\Gamma^{1}=\displaystyle\int_{\Omega}\left\vert H_m(\max(G_j,-N))-H_m(\max(u,-N))\right\vert$ and \\ $\Gamma^{2}=\displaystyle\int_{\Omega}\left\vert H_m(\max(u,-N))-H_m(u)\right\vert.$ Then
\begin{eqnarray*}
\tau_j &=& \int_{\Omega}\left\vert H_m(G_j)-H_m(\max(G_j,-N))\right\vert+\Gamma^{1}+\Gamma^{2}\\
&=&\int_{\{G_j<-N\}}H_m(G_j)+\int_{\{G_j\geq-N\}}\underbrace{\left\vert H_m(G_j)-H_m(\max(G_j,-N))\right\vert}_{=0}+\Gamma^1+\Gamma^2\\
&\leq &\dfrac{1}{N}\int_{\Omega}(-u_j)H_m(u_j)+\Gamma^1+\Gamma^2\\
&\leq & \dfrac{C}{N}+\Gamma^{1}+\Gamma^{2},
\end{eqnarray*} 
where $C=\int_{\Omega}(-u_j)H_m(u_j)<+\infty.$ On the other hand, it follows from the construction of $G_j$ that $\max(G_j,-N))$ converges to $\max(u,-N)$ in $C_m$-capacity, then by Corollary \ref{00} we conclude that 
$$H_m(\max(G_j,-N))\;\hbox{converges\;weakly\;to}\;H_m(\max(u,-N)),\;\;\hbox{as}\; j\rightarrow+\infty,$$ 
Then $\Gamma^1\rightarrow 0$ as $j\rightarrow +\infty.$ Furtheremore, the same argument gives that $H_m(\max(u,-N))$ converges\;weakly\;to $H_m(u)$, as $N\rightarrow+\infty,$ so $\Gamma^2\rightarrow0$, as $N\rightarrow +\infty,$ which completes the proof.
\end{proof}
\begin{thm}\label{1004}
Let $\Omega$ be a smoothly bounded strictly $m$-pseudoconvex domain in $\mathbb{C}^{n}$, $n\geq2$, $p\geq1$, $\mu$ a positive measure on $\Omega$ with finite mass and $H\in \mathcal{MSH}_{m}(\Omega)\cap C^{\infty}(\partial \Omega)$. Then there is a uniquely determined $u\in \mathcal{F}_{m}^{p}(H)$ with $H_m(u)=\mu$ if and only if $\mu$ satisfies (\ref{1002}).
\end{thm}
\begin{proof}
Suppose $u \in\mathcal{F}_{m}^{p}(H),$ $H_m(u)=\mu.$ Then $H\geq u\geq \varphi+H$ for some $\varphi\in\mathcal{F}_{m}^{p},$ so $0\geq u+H^{-}\geq \varphi+H+H^{-}$. By \cite[Theorem 3.9]{Chi2} $u+H^{-}\in\mathcal{F}_{m}^{p}$, since  $\varphi+H+H^{-}\in\mathcal{F}_{m}^{p}$. Therefore, by \cite[Proposition 5.4 $\&$ Theorem 6.2]{Chi2}, $H_m(u+H^{-})$ satisfies (\ref{1002}), and since $\mu=H_m(u)\leq H_m(u+H^{-})$, so does $\mu$. We split the proof into two steps.
\par \textbf{Step 1.} Assume that $p>m/(m-1)$ and $\mu$ has compact support. Let us consider a subdivision $I_s$ of supp$\mu$ consisting of $3^{2ns}$ congruent semi-open cubes $I^{j}_{s}$ with side $d_s = d/3^{s}$,  where $d:= \hbox{diam}(\Omega)$
and $1\leq j\leq 3^{2ns}$. Let $\psi$ be as in \cite[Theorem 4.1]{ACH}, since $\mu$ is compactly supported, then we can find $0<A$ big enough such that $H_m(A\psi)\geq \mu$. Let $[\Omega_{k}]_{k=1}^{\infty}$ be an increasing sequence of smoothly bounded strictly $(m-1)-$pseudoconvex domains. Then \cite[Theorem $2.10$]{Di-Ko} provides solutions $u^{k}_s=U_m(\mu_{\vert_{\Omega_{k}}},0)$, such that 
$$H_m(u^{k}_s)=\mu_s:=\sum_{j}\left(\int_{I^{j}_{s}}d\mu\right)\chi_{I^{j}_{s}}\dfrac{1}{d^{2n}_{s}}dV\; \hbox{in}\;\; \Omega_k,$$
with $u^{k+1}_{s_{\vert_{\partial\Omega_{k+1}}}}=0,$ $A\psi_{\vert_{\partial\Omega_{k+1}}}<0$, and $H_m(A\psi)\geq H_m(u_{s}^{k+1})$, then by \cite[(c) in Corollary 1.15]{Chi2} we get 
\begin{equation}\label{1005}
A\psi\leq u^{k+1}_{s}\leq u_{s}^{k}\leq 0\; \hbox{on}\; \Omega_k.
\end{equation}
Now, let prove that the sequence $[u_{s}^{k}]$ is locally uniformly convergent on $\Omega$. Take $K\Subset \Omega, \varepsilon>0$ and let $k_0$ be such that $K\subset\Omega_{k_0}$  and $\Vert A\psi\Vert_{\partial\Omega_{k_0}}\leq\varepsilon.$ Then by (\ref{1005}) and again the comaparison principle \cite[(c) in Corollary 1.15 ]{Chi2} for $k,j\geq k_0$ one has
$$\Vert u_{s}^{k}-u_{s}^{j}\Vert_{K}\leq\Vert u_{s}^{k}-u_{s}^{j}\Vert_{\Omega_{k_0}}\leq\Vert A\psi\Vert_{\Omega_{k_0}}\leq\Vert A\psi\Vert_{\partial\Omega_{k_0}}\leq\varepsilon,$$
thus the sequence $[u_{s}^{k}]_{k}$ is locally uniformly convergent on $\Omega$. Therefore the function $u_{s}:=\displaystyle\lim_{k\rightarrow+\infty}u_{s}^{k} \in \mathcal{E}_m^{0}(\Omega)\cap C(\overline{\Omega}),$ and since $[u_{s}^{k}]_{k}$ decreases to $u_{s}$, we obtain 
$$H_m(u_s)=\mu_s:=\sum_{j}\left(\int_{I^{j}_{s}}d\mu\right)\chi_{I^{j}_{s}}\dfrac{1}{d^{2n}_{s}}dV\;\hbox{in}\;\; \Omega,$$
which converges weakly to $\mu$. Using the super mean-value proprety for $m$-superharmic functions, we have
$$\int (-u_s) H_m(u_s)\leq C\int (-u_s)d\mu,$$
which is uniformly bounded since $\displaystyle\int H_m(u_s)<+\infty,$ as already noted in (\ref{1011}). 
Let $\Omega^{'}\Subset\Omega,$ and let $W$ be a smoothly strictly $m$-pseudoconvex domain such that $\Omega^{'}\Subset W\Subset\Omega,$ then ,thanks to [DK, Theorem 2.10] one can find $h\in\mathcal{SH}_m
(\Omega)\cap C(\overline{\Omega})$ such that $H_m(h)=gdV$, where $g:= \chi_{\Omega^{'}}.$ On the other hand, by \cite[Lemma 3.5]{Chi2} we have 
$$\sup_s\int_{\Omega^{'}}(-u_s)dV=\sup_s\int_{\Omega^{'}}(-u_s) H_m(h)\leq D_1\sup_s\int_{\Omega^{'}}(-u_s)H_m(u_s)\times \int_{\Omega^{'}}(-h)H_m(h)<+\infty,$$
so we can pick a subsequence $[u_{s_{j}}]_{j=1}^{\infty}$, wich we denoted again $[u_s]$, such that $u_s \rightarrow u\in \mathcal{SH}_m(\Omega),$ as $s\rightarrow+\infty,$ a.e dV, and again since 
$H_m(u_{j}^{s})\leq H_m(U_m(\mu_{j}^{s},0)+H)$, Then by Corollary \ref{14}
$H\geq u^{s}_{j}\geq U_m(\mu^{s}_{j},0)+H.$
$$H\geq u_s\geq U_m(H_m(u_s),0)+ H,$$
where $\limsup_{s\rightarrow+\infty}U_m((H_m(u_s),0))=\omega\in\mathcal{F}_m^{p}.$ Therefore, $u\in \mathcal{F}_m^{p}(H).$ Let's define now
$$\Theta_s(x)=\dfrac{1}{B(nd_s)}\int_{\{\vert\xi\vert<nd_s\}}\vert u(x+\xi)-u_s(x+\xi)\vert dV,$$
where $B(r)$ is the volume of the ball with radius $r.$ Then 
$$\int_{\Omega}\vert u-u_s\vert H_m(u_s)=\sum_{j}\left(\int_{I^{j}_{s}}d\mu\right)\dfrac{1}{d^{2n}_{s}}\int_{I^{j}_{s}}\vert u-u_s\vert dV\leq\sum_{j}\dfrac{B(nd_s)}{d^{2n}_{s}}\left(\int_{I^{j}_{s}}\Theta_s(x)d\mu(x)\right)\leq C\left(\int_{I^{j}_{s}}\Theta_s(x)d\mu(x)\right).$$
Now, denote by $\zeta= x+\xi$ and $U_j(\zeta):=\sup_{j\geq s}u_j(\zeta).$ We have
\begin{eqnarray*}
\Theta_s(x)&=& \dfrac{1}{B(nd_s)}\int_{\{\vert\xi\vert<nd_s\}}\vert u(\zeta)-u_s(\zeta)\vert dV\\
&=&\dfrac{1}{B(nd_s)}\int_{\{\vert\xi\vert<nd_s\}}\vert u(\zeta)-U_j(\zeta)+U_j(\zeta)-u_s(\zeta)\vert dV\\
&\leq& \dfrac{1}{B(nd_s)}\int_{\{\vert\xi\vert<nd_s\}} (U_j(\zeta)-u(\zeta))dV+\int_{\{\vert\xi\vert<nd_s\}} U_j(\zeta)dV-\int_{\{\vert\xi\vert\leq nd_s\}} u_s(\zeta) dV\\
&\leq &\dfrac{1}{B(nd_s)}\int_{\{\vert\xi\vert<nd_s\}} [(U_j(\zeta))^{*}-u(\zeta)]dV+\int_{\{\vert\xi\vert<nd_s\}} (U_j(\zeta))^{*}dV-u_s(x).
\end{eqnarray*} 
Then it follows from monotone convergence and Lemma \ref{1006} that $\int_{\Omega}\Theta_s(x)d\mu(x)\rightarrow 0$, as $s\rightarrow+\infty$, and using Lemma \ref{1007} we obtain that $H_m(u_s)$ tends weakly to $H_m(u)$, as $s\rightarrow+\infty$.
\par \textbf{Step 2.} Now, assume that $p\geq 1$. Let $[K_j]$ be an increasing sequence of compact subsets of $\Omega$ with $\cup^{\infty}_{j=1}K_j=\Omega$. By Theorem \ref{51} there exists $\psi_j\in \mathcal{E}_m^{0}$ such that $\chi_{K_j}d\mu=g_j H_m(\psi_j)$ for some $0\leq g_j^{1}(H_m(\psi_j))$. Let $\mu_j^{s}:=\inf(g_j,s)H_m(\psi_j),$ then by Theorem \ref{51} there exists $\varphi\in\mathcal{E}_m^{0}\subset \mathcal{F}_m^{p}$ such that $H_m(\varphi)=\mu_j^{s}$, this implies by \cite[Theorem 6.2]{Chi2} that $\mu^{s}_{j}$ satisfies (\ref{1002}) for $p>m/(m-1)$, therefore, We have by the first part that there exists  $u^{s}_{j}\in\mathcal{E}^{0}_{m}(H)$ with $H_m(u^{s}_{j})=\mu_j^{s}$, and $H\geq u^{s}_{j}\geq U_m(\mu^{s}_{j},0)+H.$
However, $H_m(U_m(\mu,0))\geq H_m(U_m(\mu^{s}_{j},0)),$ then $U_m(\mu^{s}_{j},0)\geq U_m(\mu,0)$, 
Therefore 
$$H\geq u^{s}_{j}\geq U_m(\mu,0)+H.$$
Since $U_m(\mu,0)\in\mathcal{F}_m^{p}$ (thanks to \cite[Theorem 6.2]{Chi2}), so $u_j:=\displaystyle\lim_{s\rightarrow+\infty}u^{s}_{j}\in \mathcal{F}_m^{p}(H)$, and finally $u:=\lim_{j\rightarrow+\infty}u_{j}\in \mathcal{F}_m^{p}(H)$. In the other hand, since $H_m(u_j)=\chi_{K_j}d\mu$, it follows that $H_m(u)=\mu$.
\end{proof}
\subsection{The Dirichlet problem with continuous boundary data}
\par We start this section by recalling the decomposition theorem for positive measures. For the proof of Theorem \ref{51} we use the Radon-Nikodym theorem in \cite{Ceg1} and the same arguments in \cite[Theorem 5.3]{Chi2}.
\begin{thm}\label{51}
Assume that $\mu$ is a positive measure on  $\Omega$, Then there exists a function  $\phi \in \mathcal{E}_{m}^{0}$ and a function $0\leq f \in L^{1}_{Loc}\left(H_m(\phi)\right)$ such that $\mu= fH_m(\phi)+ \nu,$
where $\nu$ is carried by an $m$-polar set.
In particular, if $\mu$ vanishes on all $m$-polar sets, then there is an increasing sequence of measures $H_m(u_j)$ tending to $\mu$ as $j\rightarrow+\infty$, where $u_j\in \mathcal{E}_{m}^{0}$.
\end{thm}
\begin{thm}\label{1009}
Suppose that $\Omega$ is a bounded $m$-pseudoconvex domain, $H\in \mathcal{MSH}_{m}(\Omega)\cap C(\partial \Omega),$ and that $\mu$ is a positive measure on $\Omega$ such that $U_m(\mu,H)\in\mathcal{SH}_{m}(\Omega)\cap L^{\infty}(\Omega)$ and such that $\lim_{z\rightarrow \xi}U_m(\mu,H)(z)=H(\xi)$, $\forall \xi\in \partial\Omega$. Then for every positive measure on $\nu$ dominated by $\mu$, $H_m(U_m(\nu,H))=\nu$ and $U_m(\nu,H)$ satisfies the inequality $H\geq U_m(\nu,H)\geq U_m(\mu,H).$
\end{thm}
\begin{proof}
Since $\nu$ can be approximated by an increasing sequence of compactly supported measures, so, it is no restriction in considering $\nu$ compactly supported. Assume first that $\Omega$ is smoothly bounded strictly $m$-pseudoconvex, that $H\in C^{\infty}(\partial \Omega)$ and that $0\leq\nu\leq\mu$, then $B_m(\nu,H)\neq\varnothing$, so $U_m(\nu,H)+H^{-}\in\mathcal{B}_m(\nu,0),$ hence $\mathcal{B}_m(\nu,0)\neq\varnothing,$ this implies that $U_m(\nu,0)\in \mathcal{E}_{m}^{0}\subset\mathcal{E}_{m}^{p}$ and $\nu\leq H_m(U_m(\nu,0))$, however, by \cite[Lemma 3.5]{Chi2} $U_m(\nu,0)$ satisfies (\ref{1002}) for any $1\leq p\leq+\infty$, so is $\nu$. By Theorem \ref{1004} there is a uniquely determined function $\vartheta$, namely $\vartheta=U_m(\nu, H)$ and $\lim_{z\rightarrow \xi}\vartheta(z)=H(\xi)$, $\forall \xi\in \partial\Omega$.
\par Assume now that $\Omega$ is $m$-pseudoconvex and let $[\Omega_{k}]_{k=1}^{\infty}$ be an increasing sequence of smoothly bounded strictly $m$-pseudoconvex domains with $\cup^{\infty}_{k=1}\Omega_{k} =\Omega$, where supp$\nu\Subset\Omega_1$. Since each $H_k:=H\vert_{\partial\Omega_{k}}=U_m(0,H)\vert_{\partial\Omega_{k}}$ is upper semicontinuous, there exists $H_{k,j}\in C^{\infty}(\partial \Omega_{k})$ with $H_{k,j}\searrow H_k$, as $j\rightarrow+\infty.$ By the first part of the proof, there exists a uniquely determined functions $u_{k,j}\in\mathcal{SH}_{m}(\Omega)\cap L^{\infty}(\Omega_{k})$ with $H_m(u_{k,j})=\nu=U_m(\nu,H_{k,j})$ and $\lim_{z\rightarrow \xi}u_{k,j}(z)=H_{k}(\xi)$, $\forall \xi\in \partial\Omega_{k}$. Therefore, $U_m(\nu,H)\vert_{\Omega_{k}} \in \mathcal{B}_{m}(\nu,H_{k,j})$, so $U_m(\nu,H)\vert_{\Omega_{k}}\leq u_{k,j}.$ Since $u_{k,j}\searrow u_k$, as $j\rightarrow+\infty$, by continuity of the Hessian Operator $H_m(\cdot)$ for a decreasing sequences  we have $H_m(u_k)=\nu$ and $$U_m(\mu,H)\leq U_m(\nu,H)\leq u_k\leq U_m(0,H_k)=U_m(0,H)\vert_{\Omega_{k}}\; \hbox{on}\;\Omega_k.$$ 
where the last equality follows from the comparison principle. 
\par Finally, $u_{k+1}\vert_{\partial\Omega_{k}}\leq U_m(0,H_{k+1})\vert_{\partial\Omega_{k}}= U_m(0,H)\vert_{\partial\Omega_{k}}=H_k$, so $[u_{k}]_{k=1}^{\infty}$ is a decreasing sequence; then by leting $k\rightarrow+\infty$ the proof of the theorem is complete.
\end{proof}
\begin{thm}\label{121}
Let $\Omega\subset \mathbb{C}^{n}$ be a bounded $m$-hyperconvex domain. Assume that $\mu$ is a positive measure which vanishes on all $m$-polar subsets of $\Omega$ and with bounded total mass. Then for every $H\in\mathcal{MSH}_{m}(\Omega)\cap C(\overline{\Omega})$ there exists a uniquely determined function $u\in\mathcal{F}_{m}^{a}(H)$ with $H_{m}(u)=\mu.$
\end{thm}
\begin{proof}
The case $H=0$ is \cite[Theorem $1.2$]{Chi2}. For the general case, we proceed with the same argument as in \cite[Theorem $3.4$]{A}. We start with the existance part. Since $\mu$ vanishes on $m$-polar sets and has a finite total mass, it follows from Theorem \ref{51} that there exist functions $\phi\in \mathcal{E}_{m}^{0}$ and $0\leq f\in L^{1}_{Loc}\left(H_m(\phi)\right)$, such that $\mu=f H_m(\phi)$. For each $j\in\mathbb{N}$, let $\mu_{j}:=\min\{\phi,j\}H_m(\phi)$. We have $\mu_j\leq H_m(j^{1/ m}\phi)$, so by \cite{Ngo}, there exists a uniquely determined function $\psi_j\in\mathcal{E}_{m}^{0}$ such that 
\begin{equation}\label{1008}
H_m(\psi_j)=\mu_j.
\end{equation}
 This construction implies that $[\psi_j]_{j}$ is decreasing, that $\psi_j+H\in \mathcal{SH}_{m}(\Omega)\cap L^{\infty}(\Omega)$, that $\displaystyle\lim_{z\rightarrow \xi}(\psi_j+H)=H(\xi)$ for every $\xi\in \partial \Omega$, and that $U_m(H_m(\psi_j+H),H)=\psi_j+H.$ Equality (\ref{1008}) implies that $H_m(\psi_j+H)\geq \mu_j$. Therefore, by Theorem \ref{1009} there exists $U_j:=U_m(\mu_j,H)$ such that 
 \begin{equation}\label{1010}
H\geq U_j\geq \psi_j+H.
\end{equation}
Namely, $U_j\in\mathcal{E}_{m}^{0}(H)$, and $[U_j]_j$ is a decreasing sequence. Since $\mu(\Omega)<+\infty$ by assumption, it follows that
$$\sup_j\int_{\Omega}H_m(\psi_j)=\sup_j\int_{\Omega}H_m(U_j)\leq \sup_j\mu_j(\Omega)\leq \mu(\Omega)<+\infty,$$
so $\psi:=\lim_{j\rightarrow+\infty}\psi_j\in \mathcal{F}_m^{a}$. Now let $u:=\lim_{j\rightarrow+\infty}U_j$, then $u\in\mathcal{SH}_m(\Omega)$ and by letting $j\rightarrow+\infty$ in (\ref{1010}) we get $H\geq u\geq \psi+H.$ and the proof is completed.
\end{proof}
\begin{lem}\label{102}
If $H\in\mathcal{MSH}_{m}\cap L^{\infty}(\Omega)$ and $\phi\in\mathcal{E}^0_m$ with supp$(H_m(\phi))\Subset\Omega$. Then
$\exists$ $u\in\mathcal{E}^0_m(H)$ such that 
$$H_m(u)=H_m(\phi).$$
\end{lem}
\begin{proof}
Let $H_k,\phi_k\in\phi\in\mathcal{E}^0_m\cap C(\overline{\Omega})$ such that $H_k$ decreases to $H$ and $\phi_k$ to $\phi$, and let $[\Omega_j]$ be a fondamental sequence such that Supp$(H_m(\phi))\Subset\Omega_1\Subset\ldots\Subset\Omega_j\Subset\ldots\Subset\Omega$. Note firstly that by Theorem \ref{121} there exists $u_{j,k}:=U_m(\mu\vert_{\Omega_j},\phi_k+H^{j}_k).$ Define
\begin{center}
$\left\{
  \begin{array}{ll}
    \Psi_{j,k}=\max(u_{j,k},\phi_k+H^{j}_k), & \hbox{on}\; \overline{\Omega}_j, \\
    \Psi_{j,k}=\phi_k+H^{j}_k, &\hbox{on}\;\Omega\setminus\overline{\Omega}_j.
  \end{array}
\right.
$
 \end{center}
Then $\Psi_{j,k}\in \mathcal{SH}^{-}_{m}(\Omega)$ and  $H^{j}_k\geq \Psi_{j,k}\geq H^{j}_k+\phi_k.$ Let $\psi_j:=\displaystyle\lim_{k\rightarrow+\infty}\Psi_{j,k}.$
Since $\Psi_{j,k+1}\leq\Psi_{j,k}$ and $u_{j,k+1}\leq u_{j,k}$ on $\Omega_j,$ then $H\geq \psi_{j}\geq H+\phi.$
Now define $u_j:=\displaystyle\lim_{k\rightarrow+\infty}u_{j,k}$ on $\Omega_j$,
We have the following key facts:
\textit{FACT 1}. $u_j=\psi_j$ on $\Omega_j.$ Indeed: it is clear that $\psi_j\geq u_j$, however we have that
$$\left\{
  \begin{array}{ll}
    \displaystyle\limsup_{z\rightarrow \xi}(\phi+H)\leq u_{j,k}=\phi_k+H_k, & \hbox{on}\; \partial\Omega_j, \\
    H_m(\phi+H)\geq H_m(u_{j,k})=H_m(\phi), &\hbox{on}\;\Omega_j
  \end{array}
\right.
$$
Hence, $u_j\geq \phi+H$ on $\Omega_j$ and the \textit{FACT 1} is proved, In particular, $H_m(\psi_j)|_{\Omega_j}=H_m(\phi)$\\
\textit{FACT 2}. $\psi_{j+1}\geq\psi_{j}$ on $\Omega$ indeed,
$$\left\{
  \begin{array}{ll}
    \psi_{j+1}= \psi_{j}=\phi+H, & \hbox{on}\; \Omega\setminus\Omega_{j+1}, \\
    \psi_{j+1}= u_{j+1}, &\hbox{on}\;\Omega_{j+1}\;\;\; \hbox{(by\;the\;\textit{FACT 1})}
  \end{array}
\right.
$$
So we have $H_m(\psi_{j+1})=H_m(\phi)$ on $\Omega_{j+1}$, $H_m(\psi_j)\geq H_m(\phi)$ on $\Omega_{j+1},$ and
$\psi_{j+1}=u_{j+1}\geq\phi+H=\psi_j$ on $\Omega_{j+1}\cap\Omega\setminus\Omega_j.$
Hence, $\psi_{j+1}\geq \psi_{j}$ on $\Omega_{j+1}$, so $\psi_{j+1}\geq \psi_{j}$ on $\Omega$. Put now $u=(\displaystyle\lim_{j\rightarrow+ \infty}\psi_j)^{*}$. Then $H_m(u)=H_m(\phi)$ and $H\geq u\geq H+\phi.$
\end{proof}
\begin{thm}\label{21}
Suppose $\mu=H_m(v)$ where $v\in \mathcal{N}_m^a$. Then, to every $H\in\mathcal{MSH}_{m}\cap L^{\infty}(\Omega)$, there is a uniquely determined function $\varphi\in \mathcal{N}^a_m(H)$ with $\mu=H_m(\varphi)$.
\end{thm}
\begin{proof}
By Theorem \ref{51} there exists $g\in\mathcal{E}^0_m$ and a function $0\leq f \in L^{1}_{Loc}\left(H_m(\phi)\right)$ such that $\mu= fH_m(g).$
Consider $\mu_k=\chi_{\Omega_{k}}\inf(f,k)H_m(g)$ then thers exists  $\psi_k\in\mathcal{E}^{0}_{m}\subset\mathcal{N}^{a}_{m}$ such that $H_m(\psi_k)=\mu_k,$ then by Lemma \ref{102} we can find $u_k\in\mathcal{E}_{m}^{0}(H)$ such that $H_m(u_k)=\mu_k,$
where $[u_k]_k$ is a decreasing sequence with $H\geq u_k\geq \psi_k+H$. Sine $\psi_k\geq v$ (by Corollary \ref{14}), then $H\geq u_k\geq v +H$ This estabilish the existance part of a solution. the uniquenes is due to Corollary \ref{14}.
\end{proof}
We give now a solution for the Dirichlet problem for measure which is written as the Hessian measure of function belonging to class $\mathcal{N}_m^{a}$. 
\begin{thm}\label{52}
Assume that $\mu$ is a non-negative measure defined on $\Omega$ by $\mu=H_m(\varphi)$, $\varphi \in \mathcal{N}^a_m$. Then for every function $H \in \mathcal{E}_m$ such that $H_m(H)\leq \mu,$ there exists a  uniquely determined function  $u\in \mathcal{N}_m(H)$ such that $H_m(u)=\mu$ on $\Omega$.
\end{thm}
\begin{proof}
The uniqueness is due to Corollary \ref{14}. For existance we proceed as follows:
Since $H \in\mathcal{E}_m,$ \cite[Theorem 3.1]{Chi2} implies that there exists a decreasing sequene $[H_k], H_k \in \mathcal{E}_m^0\cap C(\bar\Omega)$, that converges pointwise to  $H$, as $k\rightarrow +\infty.$
Let $[\Omega_{j}]$ be the fondamental sequence of $\Omega$. For each  $j, k \in\mathbb{N}$, define $H^j_k$ as in Definition \ref{120}, then $H^j_k\in \mathcal{E}_m^0$ and $H^j_k$ is $m$-maximal on $\Omega_j.$ Let  $\mu_j=\chi_{\Omega_j}\mu,$ for each $j\in \mathbb{N},$  $\mu_j$ is a Boral measure compactly supported on $\Omega$ that puts no mass on  $m$-polar  sets and we have that $\mu_j(\Omega_j)<\mu_j(\Omega)<+\infty.$ Hence,  by  \cite[Theorem 1.2]{Chi2} there exists a uniquely determined function $\varphi_j\in\mathcal{F}^{a}_m(\Omega_j)\subset\mathcal{N}^{a}_m(\Omega_j)$ such that
$H_m(\varphi_j)=\mu_j$ on $\Omega_j$. Moreover, from Theorem \ref{21}, there exists a sequence $u_{j,k} \in \mathcal{F}^{a}_m(\Omega_j,H^j_k)$
such that $H_m(u_{j,k})=\mu_j$ on $\Omega_j.$ However
  $H_m(u_{j,k})= H_m(\varphi_j)\leq H_m(\varphi_j+H_k^j),$
then by Corollary \ref{14} we have that
$
   H^j_k\geq u_{j,k}\geq \varphi_j+H^j_k \;\;\; \hbox{on}\;\;\Omega_j,
$
and $[u_{j,k}]^\infty_{k=1}$ is a decreasing sequence. Let $k\rightarrow+\infty,$ then $H^j\geq u_{j}\geq \varphi_j+H^j$ on $\Omega_j,$
$(i.e)$ $u_j \in \mathcal{F}_m(\Omega, H^j)\subset \mathcal{N}_m(\Omega, H^j)$. Since $H_m(H)\leq \mu,$ we get
        $H_m(u_j)=\mu_j=\chi_{\Omega_j}\mu\geq H_m(H)=0 \;\; \hbox{on}\;\; \Omega_j,$
Thus, Corollary \ref{14} implies again that $u_j\leq H$ on $\Omega_j.$  On the other hand, by the construction of $\mu_j$ and by the fact that $[\Omega_j]$ is an increasing sequence  we have that
$H_m(u_j)=H_m(u_{j+1})$ on $\Omega_j$. Then  $[u_j]$ is decreasing and we get that $H\geq u_j\geq \varphi_j+H^{j}\geq \varphi_j+H$ on $\Omega_j.$
Thus,  $u:=(\displaystyle\lim_{j\rightarrow+\infty}u_j) \in \mathcal{N}_m(\Omega, H)$ with $H_m(u)=\mu.$
\end{proof}
\par The next Proposition will be useful in section 4. Note that if $\mu=H_m(v)$ for some $v\in\mathcal{E}_m^{a},$ then we define $U_m(\mu,0)$ to be $\lim_jU_m(\chi_{\Omega_j}\mu,0),$ in this case we affirm that $U_m(\mu,0)\in\mathcal{N}_m^{a},$ and $U_m(\mu,0)\geq v.$  
\begin{pro}\label{1013}
Let $\mu$ be a non-negative measure defined on $\Omega$ such that it vanishes on $m$-polar sets of $\Omega$ and that there exists a function $\varphi\in \mathcal{E}_m$, $\varphi\neq 0$ such that $\displaystyle\int_\Omega\varphi d\mu>-\infty.$ Then $$U_m(\mu,0)\in \mathcal{N}^{a}_{m} \;\;\hbox{and}\;\; H_m(U_m(\mu,0))=\mu.$$
\end{pro}
\begin{proof}
We split the proof into two steps.
\par \textit{\textbf{Step 1.}}\; We shall firstly prove that $U_m(\mu,0)\in\mathcal{E}_m$. Indeed, Let $\mu_j=\chi_{\Omega_j}\mu$, for each $j\geq 1$, then it follows from \cite[Theorem 1.2]{Chi2} that there exists $U_m(\mu_j,0)\in\mathcal{F}_{m}^{a}\cap L_{loc}^{\infty}(\Omega)$ such that $H_m(U_m(\mu_j,0))=\mu_j$. Note that  $\mu_j\nearrow\mu$ weakly, as $j\rightarrow+\infty$ and $U_m(\mu_j,0)$ decreases to  some function $u\in \mathcal{SH}_m(\Omega)$ as $j\rightarrow+\infty.$ By hypothesis we have 
\begin{equation}\label{1111}
\sup_j\int (-\varphi) H_m(U_m(\mu_j,0))=\sup_j\int (-\varphi) d\mu_j\leq \int(-\varphi) d\mu<+\infty.
\end{equation}
Let $W\Subset \Omega$ and $U_j=\sup\{\psi\in\mathcal{SH}_m(\Omega):\;\psi_{\vert_{W}}\leq U_m(\mu_j,0)_{\vert_{W}}\}$. We claim that $U_j\in \mathcal{E}_m^{0}$. Indeed $U_j\geq U_m(\mu_j,0)$, then by \cite[Theorem 3.22]{Chi2}, 
$$\displaystyle\int_{\Omega}H_m(U_j)\leq \int_{\Omega}H_m((\mu_j,0))<+\infty.$$ On the other hand there exists an exaustive function $\rho$ and $A>>1$ such that $A\rho_{\vert_{W}}\leq U_{j_{\vert_{W}}}$, this implies that $\lim_{z\rightarrow\partial\Omega}U_j(z)=0$. Note also that  $U_j\searrow u$ on $W$, as $j\rightarrow+\infty$, so it follows from Lemma \ref{18}  and (\ref{1111}) that 
$$\sup_j\int (-\varphi) H_m(U_j)\leq \sup_j\int(-\varphi) H_m(U_m(\mu_j,0)<+\infty.$$ 
Since Supp\;$H_m(U_j)\subset \overline{W}$, we have 
$$\sup_j\int H_m(U_j)\leq \left(\inf_W(-\varphi)\right)^{-1}\sup_j\int(-\varphi) H_m(U_j)<+\infty,$$
this is for all $W\Subset \Omega$, therefore $u\in\mathcal{E}_{m}.$ Hence by continuity of the Hessian Operator for a decreasing sequences we have $H_m(U_m(\mu_j,0))\rightarrow H_m(u)$ weakly, as $j\rightarrow+\infty$ 
then $\mu=H_m(u)$ and since $\mu$ vanishes on m-polar sets then $u\in \mathcal{E}_{m}^{a}$, and $u=U_m(\mu,0)$
\par \textit{\textbf{Step 2.}} Now we prove that $U_m(\mu,0)\in \mathcal{N}_{m}^{a}$. We have
$U_m(\mu_j,0)+U_m((1-\chi_j)\mu,0)\leq U_m(\mu,0).$
Indeed, since $H_m(U_m(\mu_j,0))+H_m(U_m((1-\chi_j)\mu,0))=\mu,$ then 
$$H_m\left[U_m(\mu_j,0)+U_m((1-\chi_j)\mu,0))\right]\geq \mu=U_m(\mu,0),$$
then by the comparison principle, $U_m(\mu_j,0)+U_m((1-\chi_j)\mu,0)\leq U_m(\mu,0)$, hence, it follows from the propreties after Definition \ref{120} that 
$$U_m((1-\chi_j)\mu,0)\leq\widetilde{\mathit{U_m(\mu_j,0)}}+\widetilde{\mathit{U_m((1-\chi_j)\mu,0)}}\leq \widetilde{\mathit{U_m(\mu,0)}},$$
note that $\widetilde{\mathit{U_m(\mu_j,0)}}=0$ since $U_m(\mu_j,0)\in\mathcal{F}_{m}^{a}$. Therefore it follows from Lemma \ref{80} that 
 $$0\leq\int (-\widetilde{\mathit{U_m(\mu,0))}}^{m}H_m(\varphi)\leq \int(-U_m((1-\chi_j)\mu,0))^{m}H_m(\varphi)\leq m!\int(-\varphi)H_m[U_m((1-\chi_j)\mu,0)].$$
 By letting $j\rightarrow+\infty$ we derive that $\int (-\widetilde{\mathit{U_m(\mu,0)}})H_m(\varphi)=0,$ so $\widetilde{\mathit{U_m(\mu,0))}}=0$ and $U_m(\mu,0))\in\mathcal{N}^{a}_{m}.$
\end{proof}
\section{Hessian measures carried on $m$-polar sets}\label{77773}
\subsection{Some auxiliary results}
\begin{defi}\label{24}
Let $u \in \mathcal{E}_m$ and $0\leq \tau$ be a bounded lower semicontinuous function. We define
\begin{equation*}
u_{\tau}:=\sup\underbrace{\{\phi\in \mathcal{SH}_m(\Omega): \phi\leq \tau^{1/m}u\}}_{\mathcal{A}}.
\end{equation*}
\end{defi}
We have the following elementary propreties:
\begin{description}
  \item[$(1)$] If $u, v \in \mathcal{E}_m$ such that $u\leq v$, then $u_{\tau}\leq v_{\tau}$.
  \item[$(2)$] if $u\in \mathcal{E}_m,$ then $u_{\tau}\in \mathcal{E}_m$, since $\|\tau\|^{1/m}_{L^{\infty}(\Omega)}u\leq u_{\tau}\leq 0.$
  \item[$(3)$] If $\tau_1, \tau_2$ are bounded lower semicontinuous functions such that $\tau_1\leq \tau_2,$ then $u_{\tau_1}\geq u_{\tau_2}.$
  \item[$(4)$] If $u\in \mathcal{E}_m,$ then $\hbox{Supp}\left(H_m(u))\right)\subseteq \hbox{Supp}\tau$, and if $\hbox{Supp}\;\tau$ is compact then $u_{\tau} \in \mathcal{F}_m.$ Indeed, if we take $\mathcal{D}\Subset \Omega$ such that $\hbox{Supp}\;\tau\subset \mathcal{D},$  since $u_{\tau} \in \mathcal{E}_m$ then there exists $v \in \mathcal{F}_m(D)$ such that $u_{\tau}=v$ on $\mathcal{D},$ and $v\leq \tau^{1/m}u$ on $\mathcal{D}$ and $v<0=\tau^{1/m}u$ on $\Omega\setminus\mathcal{D}$ Hence $v\leq \tau^{1/m}u$ on $\Omega.$ Then $v \in \mathcal{A},$ this implies that $v\leq u_{\tau},$ then by [Ch2, Theorem 3.9] we get that  $u_{\tau}\in \mathcal{F}_m.$
To prove that $\hbox{Supp}\left(H_m(u)\right)\subset \hbox{Supp}\;\tau$, we take  $z_0 \in \Omega \backslash \hbox{Supp}\;\tau.$ Then, $u_\tau<0,$ (because $u_\tau(z_0)\leq\tau^{1/m}u(z_0)=0$). On the other hand, there exists $r>0$ such that $B(z_0,r) \subset \Omega\setminus\hbox{Supp}\;\tau$ and $\displaystyle\sup_{B(z_0,r)} u_{\tau}<0,$
    we claim that $\exists \hat{u}_{\tau}\in \mathcal{SH}_m(\Omega)$ such that $\hat{u}_{\tau}=u_{\tau}$ on $\Omega\backslash B(z_0,r)\supset\hbox{Supp}\tau,$ $\hat{u}_{\tau}\geq u_{\tau}$ on  $\Omega$ and $H_m(\hat{u}_{\tau})=0$ on $B(z_0,r),$ and we have
\begin{equation*}
\sup_{\partial B(z_0,r)} \hat{u}_{\tau}=\sup_{\partial B(z_0,r)} u_{\tau}=\sup_{B(z_0,r)} u_{\tau}<0,
\end{equation*}
Since $z_0$ is arbitrary, then  $\hat{u}_{\tau}<0$ on $\Omega\backslash\hbox{Supp}\tau$,and we derive that $\hat{u}_{\tau}\leq u_{\tau}$ on $\Omega$, so $\hat{u}_{\tau}=u_{\tau}$, but $H_m(\hat{u}_{\tau})=0$ on $B(z_0,r)\subset \Omega \backslash \hbox{Supp}\tau$.
  \item[$(5)$] If $[\tau_j]$, $0\leq \tau_j$ is an increasing sequence of bounded lower semicontinuous functions that converges pointwise to
a bounded lower semicontinuous function $\tau$, as $j\rightarrow +\infty$, then $[u_{\tau_j}]$  is a decreasing sequence that converges
pointwise to $u_{\tau}$, as $j\rightarrow+\infty.$
\end{description}
 Let $u \in \mathcal{E}_m$, set $\mu_u=\chi_{\{u=-\infty\}}H_m(u)$ and denote by $\Gamma$ the class of functions $f=\displaystyle\sum_{k=1}^{l}\alpha_k\chi_{E_k},\alpha_k>0$ where  $E_k$ are pairwise disjoint and $\mu$-measurables such that $f$ is compactly supported and vanishes outside $\{u=-\infty\}$. We write $T$ for the class of functions in $\Gamma$ where the $E_k$ are compact.
\begin{defi}\label{98}
Let $u\in \mathcal{E}_m$ and $0\leq g\leq 1$ be a $\mu$-measurable function. we define:
$$u^{g}:=\inf_{T\ni f\leq g}\left(\sup\{ u_\tau:\;\; f\leq  \tau,\;\;  \tau\;\; \hbox{is bounded lower semicontinuous}\}\right)^*.$$
\end{defi}
By Definition \ref{24}, we have that $u\leq u^g\leq 0$ and if $g_1\leq g_2$, then $u^{g_1}\geq u^{g_2}$. Furthermore, if $g\in T$, then
$$u^{g}=\left(\sup\underbrace{\{ u_\tau:\;\; g\leq\tau,\;\;\tau\;\; \hbox{is bounded lower semicontinuous}\}}_{\mathcal{D}}\right)^*\in \mathcal{F}_m.$$
Because, if $g$ is compactly supported in $O$, there exists  $\tau=\chi_O$ compactly supported such that $u_{\tau}\in \mathcal{F}_m$ and $\hbox{Supp}\;g\subset O,$ in this case $g\leq \tau,$ furthermore $u_{\tau}\in\mathcal{D}.$
Hence $u_{\tau}\leq u^g,$ and by \cite[Theorem 3.9]{Chi2} we have that $u^g\in\mathcal{F}_m.$
\begin{thm}\label{36}
Let $u\in \mathcal{E}_m$ and $0\leq g\leq 1$ be a $\mu_u$-measurable function that vanishes outside $\{u=-\infty\}.$ Then $u^g \in \mathcal{E}_m$ and we have that $H_m(u^g)=gH_m(u).$
\end{thm}
\begin{proof}
See \cite[Proposition 5.8]{HP}.
\end{proof}
\begin{lem}\label{33}
Assume that $\alpha,\; \beta_1,\; \beta_2$ are non-negative measures defined on $\Omega$ which satisfy the following:
\begin{description}
  \item[(1)] $\alpha$ vanishes on all $m$-polar set of $\Omega,$
  \item[(2)] there exists an $m$-polar set $A\subset\Omega$ such that $\beta_1(\Omega\backslash A)=\beta_2(\Omega\backslash A)=0$,
  \item[(3)] for every $\rho \in \mathcal{E}_m^0\cap C\overline{\Omega})$ it holds that
       $\displaystyle\int_{\Omega}(-\rho)\beta_1\leq \displaystyle\int_{\Omega}(-\rho)(\alpha+\beta_2)<+\infty.$
\end{description}
Then we have that
$$\int_{\Omega}(-\rho)\beta_1\leq \int_{\Omega}(-\rho)\beta_2,\;\;\;\;\; \forall\rho \in \mathcal{E}^0_m\cap   C(\overline{\Omega}).$$
\end{lem}
\begin{proof}
See \cite[Lemma 4.11]{ACCH}
\end{proof}
Let $u \in \mathcal{E}_m$, then by Theorem \ref{51} there exist $\psi_u \in \mathcal{E}_m^0$ and
a function $0\leq f_u\in \mathrm{L}_{loc}^{1}\left(H_m(\psi_u)\right)$ such that $H_m(u)=\alpha_u+\nu_u,$ where $\alpha_u=f_u H_m(\psi_u)$ and $\nu_u$ is a positive measure vanishing outside some $m$-polar set $\mathrm{A}\subseteq \Omega$ $\emph{(i.e}\;\nu_u(\Omega\setminus \mathrm{A})=0$). In the next Lemma we will use the notation that $\alpha_u=f_u H_m(\psi_u)$ and $\nu_u$ referring to this decomposition.
\begin{lem}\label{35}
Let $u,v \in\mathcal{E}_m.$ If there exists a function $\phi \in \mathcal{E}_m$ such that $H_m(\phi)$ vanishes on $m$-polar  sets and if
$|u-v|\leq -\phi,$ then $\nu_u=\nu_v.$
\end{lem}
\begin{proof}
Assume first that $u,v,\phi\in\mathcal{F}_m$, and let $\Omega^{'}\Subset \Omega,$ it is enough to prove that $\nu_u=\nu_v$ on $\Omega^{'}.$ The assumption
that $|u-v|\leq -\phi,$ yields that $v+\phi\leq u$, therfore, it follows from Lemma \ref{18} that
\begin{equation}\label{2000}
\int_{\Omega}(-\rho)H_m(u)\leq \int_{\Omega}(-\rho)H_m(v+\phi)<+\infty,
\end{equation}
\hbox{where} $\rho\in \mathcal{E}_{m}^{0}$. Now Put $T=(dd^cv)^{m-k}\wedge \beta^{n-m},$ since $\displaystyle\sum_{k=1}^{m}C^{k}_{m}(dd^c\phi)^{k}\wedge T\ll C_m$ we have that $\nu_{v+\phi}=\nu_v$, and
$\alpha_{v+\phi}=\alpha_v+\displaystyle\sum_{k=1}^{m}C^{k}_{m}(dd^c\phi)^{k}\wedge T.$
 Lemma \ref{33} and inequality (\ref{2000}) yields that $\displaystyle\int_{\Omega}(-\rho)\nu_v\leq \displaystyle\int_{\Omega}(-\rho)\nu_u,$
for every $\rho\in \mathcal{E}_{m}^{0}$. The same argument can be made to prove that $\displaystyle\int_{\Omega}(-\rho)\nu_u\leq \displaystyle\int_{\Omega}(-\rho)\nu_v.$ Finally it follows from \cite[Lemma 3.10]{Chi2} that $\nu_u=\nu_v$
\end{proof}
\subsection{Subsolution Theorem}
\begin{pro}\label{37}
Let $H\in \mathcal{E}_m\cap \mathcal{MSH}_m(\Omega).$
\begin{description}
  \item[(a)] If $v\in \mathcal{N}_m,$ $H_m(v)$ is carried by an $m$-polar set and 
  $\displaystyle\int_{\Omega}(-\xi)H_m(v)< +\infty$ for all $\xi \in \mathcal{E}^0_m\cap C(\overline{\Omega}),$ then
$$u:=\sup \left\{\varphi\in \mathcal{SH}_m(\Omega):\;\; \varphi\leq \min(v,H) \right\}\in \mathcal{N}_m(H)\;\; \hbox{and} \;H_m(u)=H_m(v).\;\;\;\;\;\;\;\;\;\;\;\;\;\;\;\;\;\;\;\;\;\;\;\;\;$$
  \item[(b)] Assume that  $\psi \in \mathcal{N}^a_m$ and  $v\in \mathcal{N}_m(H)$ such that $H_m(v)$ is carried by an $m$-polar set and\\
      $\displaystyle\int_{\Omega}(-\xi)(H_m(\psi)+H_m(v))<+\infty,\;\; \hbox{for\; all}\;\; \xi \in \mathcal{E}^0_m\cap C(\overline{\Omega}).$
If $u$ is the function defined on $\Omega$ by 
$$
u:=\sup\left\{\varphi\;/\; \varphi\in \mathfrak{B}(H_m(\psi),v)\right\},
$$
where 
$$\mathfrak{B}(H_m(\psi),v)=\{\varphi\in\mathcal{E}_m/\; H_m(\psi)\leq H_m(\varphi)\;\; \hbox{and}\;\; \varphi\leq v\}.
$$
Then $u\in \mathcal{N}_m(H)$ and $H_m(u)=H_m(\psi)+H_m(v).$
\end{description}
\end{pro}
\begin{proof}
\begin{description}
   \item[(a)] It is clear that $\min(v,H)$ is a negative and upper semicontinuous function, then $u\in \mathcal{SH}_m(\Omega)$ and since $H+v\leq\min(v,H)$, then $H+v\leq u\leq H$, so, $u\in\mathcal{N}_m(H)$. By \cite[Theorem 3.1]{Chi2} we can choose a decreasing sequence $[v_j],\; v_j \in \mathcal{E}^0_m\cap C(\overline{\Omega})$ that converges pointwise to $v$, as $j\rightarrow +\infty$, and by Theorem \ref{52}, there exists $\omega_j \in \mathcal{N}_m(H),$ $j\in \mathbb{N}$ such that $H_m(\omega_j)=H_m(v_j),$ then by Lemma \ref{34} $v_j\geq\omega_j.$ Let now
                $$u_j:=\sup \left\{\varphi\in \mathcal{SH}_m(\Omega):\;\varphi\leq \min(v_j,H) \right\},$$
    then $u_j\in \mathcal{E}^{0}_m(H)$ and $u_j\geq\omega_j,$  furthermore , using lemma $\ref{18}$ we get
                         $\displaystyle\int_{\Omega}(-\rho)H_m(u_j)\leq \displaystyle\int_{\Omega}(-\rho)H_m(\omega_j).$
Hence, Proposition \ref{19} implies that
    $\displaystyle\int_{\Omega}(-\rho)H_m(u)\leq \displaystyle\int_{\Omega}(-\rho)H_m(v) \; \hbox{for\;all}\;\rho \in \mathcal{E}^0_m\cap C(\overline{\Omega}),$
  hence $H_m(u)$ is carried by $\{u=-\infty\}$. On the other hand $|u-v|\leq -H,$ then by Lemma \ref{35} $H_m(u)=H_m(v).$ Thus, part (a) of this proof is completed.
  \item[(b)] By Choquet's Lemma we derive that 
  $$u:=\displaystyle\sup_j \left\{\varphi_j:\; \varphi_j\in \mathfrak{B}(H_m(\psi),v)\right\}\in\mathcal{E}_m.$$
Let $\phi_k:=\displaystyle\max_k\left\{\varphi_1...,\varphi_k\right\}$, so $[\phi_k]_k$ is an increasing sequence that converges pointwise to $\phi_k\nearrow u,$ as $k\longrightarrow +\infty$. Furthermore, we claim that $\phi_k \in \mathfrak{B}(H_m(\psi),v),$ indeed, if we take $\varphi_s,\;\varphi_r\in \mathfrak{B}(H_m(\psi),v)$, $(s,r)\in\mathrm{N}^{2}$, then by $(ii)$ in Lemma \ref{28} we have $H_m(\psi)\leq H_m(\max(\varphi_s,\varphi_r)).$ Therefore, by Corollary  \ref{00} $H_m(\psi)\leq H_m(u).$
On the other hand, by Theorem \ref{51}, there exists $\alpha_u,\;\nu_u$ two positive measures on $\Omega$ such that $H_m(u)=\alpha_u+\nu_u$ with $\alpha_u$ vanishes on an $m$-polar sets and $\nu_u$ is carried by an $m$-polar set, since $\psi+v\in \mathfrak{B}(H_m(\psi),v)$, then  $\psi+v\leq u\leq v,$ so $u\in \mathcal{N}_m(H)$.
\par Now, let's prove the second assertion. Since $\vert u-v\vert\leq -\psi,$ then by Lemma \ref{35},  $\nu_u=H_m(v)$. Note that $\alpha_u\geq H_m(\psi).$ By Proposition \ref{0}, there exists a decreasing sequence $[v_j]\subset \mathcal{E}_m^0(H),$ that  converges pointwise to $v$ as $j\rightarrow+\infty.$ Namely $\exists \gamma\in\mathcal{E}_m^0$ such that $\gamma+H\leq v_j\leq H,$ then by Lemma \ref{28} we have 
$$\int_{A} H_m(v_j)\leq \int_{A} H_m(\gamma+H)\leq C(m)\left(\int_{A} H_m(\gamma)\right)^{1/m}\left(\int_{A} H_m(H)\right)^{1/m}=0,$$
where $A$ is an $m$-polar set. Hence $H_m(v_j)$ vanishes in $m$-polar sets. Furthermore, since
$$\displaystyle\int_\Omega(-\xi)\left(H_m(\psi)+H_m(v_j)\right)<+\infty,\;\; \hbox{for\; all}\;\; \xi \in \mathcal{E}^0_m\cap C(\overline{\Omega}),$$ then  by Proposition \ref{1013} there exists a uniquely determined function $\Phi_j \in \mathcal{N}^{a}_{m}$ such that 
$$H_m(\Phi_j)=H_m(v_j)+H_m(\psi).$$ 
and using Theorem \ref{52}, one  can find $\omega_j\in \mathcal{N}_m(H)$ such that 
$H_m(\omega_j)=H_m(v_j)+H_m(\psi),$
then the comparison principle (Corollary \ref{14}) applied for $\omega_j$ and $v_j$ gives that $\omega_j\in \mathfrak{B}(H_m(\psi),v_j)$. Hence, we can consider  
 $$u_j:=\sup \left\{\varphi:\; \varphi\in \mathfrak{B}(H_m(\psi),v_j) \right\},$$
then the sequence $[u_j]$ decreases pointwise to $u$, as $j\longrightarrow+\infty$. Furthermore,  Lemma \ref{80} implies that
$$
   \int_\Omega(-\xi)H_m(u_j) \leq \int_\Omega(-\xi)H_m(\omega_j) =\int_\Omega(-\xi)\left(H_m(\psi)+H_m(v_j)\right).
$$
Then Proposition \ref{19} yields that 
$$\int_\Omega(-\xi)H_m(u)=\int_\Omega(-\xi)(\alpha_u+\nu_u)\leq\int_\Omega(-\xi)\left(H_m(\psi)+\nu_u\right).$$
for all $\xi \in \mathcal{E}^0_m\cap C(\overline{\Omega})$, since $\alpha_u\geq H_m(\psi)$, then it follows that $\int_\Omega\xi\alpha_u=\int_\Omega\xi H_m(\psi)$, hence $\alpha_u=H_m(\psi).$ Thus, this proof is completed.
\end{description}
\end{proof}
We are now in a position to prove our main Theorem. Assume that $\mu$ is a non-negative measure. It follows from Theorem \ref{51} that there exist functions  $\phi \in \mathcal{E}_{m}^{0}$, $f \in L^{1}_{Loc}\left(H_m(\phi)\right)$, $f\geq0,$ such that $\mu= fH_m(\phi)+ \nu,$
where $\nu$ is carried by an $m$-polar set.
\begin{thm}\label{1000}
 \begin{description}
 \item[$(a)$] If there exists a function  $\omega \in \mathcal{E}_m$  such that $\mu \leq H_m(\omega),$  then there exist functions $\Phi,\Psi\in \mathcal{E}_m$, $\Phi,\Psi\geq\omega$, such that
        $H_m(\Psi)=fH_m(\phi)\; \hbox{and}\; H_m(\Phi)=\nu,$
where $\nu$ is carried by $\{\Phi=-\infty\}.$
\item[$(b)$] If there exists a function $\omega \in \mathcal{E}_m$  with $\mu \leq H_m(\omega),$ then to every function
   $H\in\mathcal{E}_m\cap \mathcal{MSH}_m(\Omega),$ there exists a function $u \in \mathcal{E}_m$, $\omega+H\leq u\leq H$ with $H_m(u)=\mu.$ In particular, if $\omega \in \mathcal{N}_m$, then, $u \in \mathcal{N}_m(H).$
\end{description}
\end{thm}
\begin{proof}
Note firstly that using the Radon-Nikodym decomposition theorem we obtain that
$$fH_m(\phi)=\tau\chi_{\{\omega>-\infty\}}H_m(\omega)\;\hbox{and}\; \nu=\tau\chi_{\{\omega=-\infty\}}H_m(\omega),\;\hbox{where}\; 0\leq \tau\leq 1\;\hbox{is\;a\;borel\;function}.$$
\par $\textbf{(a)}$ Let $\mu_j=\min(f,j)H_m(\phi),$ then $\mu_j\leq H_m(j^\frac{1}{m}\phi),$ so by \cite[Theorem 2.2]{Chi2} there exists a uniquely determined function $\Psi_j\in\mathcal{E}^0_m$ such that $H_m(\Psi_j)=\mu_j$. By Corollary  \ref{14} $\Psi_j\geq \omega$ and $[\Psi_j]$ is a decreasing sequence. Therefore, $\Psi=\displaystyle\lim_{j\rightarrow+\infty} \Psi_j \in\mathcal{E}_m$ and $H_m(\Psi)=fH_m(\phi).$
On the other hand, $\tau\chi_{\{\omega=-\infty\}}H_m(\omega)$ is a positive measure vanishing outside  $\{\omega=-\infty\}$, then by  Theorem \ref{36} there exists a function $\Phi\in \mathcal{E}_m$ such that  $H_m(\Phi)=\nu$ and $\Phi\geq \omega,$ then $\textbf{(a)}$ is proved.
\par $\textbf{(b)}$ We choose an increasing sequence $[g_j]$ such that $\hbox{Supp}\;g_j \Subset \Omega,$ which converges to $g=\chi_{\{\omega=-\infty\}}\tau$, as $j\rightarrow+\infty.$ By Theorem \ref{36} $\omega^{g_j}\in \mathcal{F}_m$ and $H_m(\omega^{g_j})=g_jH_m(\omega),$ with $[\omega^{g_j}]$ is a decreasing sequence that converges  pointwise to $\omega^g,$ as $j\rightarrow +\infty.$ Moreover $\omega^g\geq \omega.$ Hence by continuity of the Hessian Operator for a decreasing sequences we have
                       $H_m(\omega)=\chi_{\{\omega=-\infty\}}\tau H_m(\omega).$
Let now
$$\mathfrak{B}(H_m(\psi_j),\min(\omega^{g_j}, H))=\{\varphi\in\mathcal{E}_m:\; H_m(\psi_j)\leq H_m(\varphi)\;\; \hbox{and}\;\; \varphi\leq \min(\omega^{g_j}, H)\},$$
and put $u_j:=\sup\left\{\varphi\;/\; \varphi\in \mathfrak{B}(H_m(\psi_j),\min(\omega^{g_j}, H))\right\},$
 then $[u_j]_j$ is a decreasing sequence that converges to some $m$-sh function $u$ as $j$ tends to $+\infty.$ On the other hand, it follows from $(a)$ in Proposition \ref{37} that we can take 
$$v_j=\sup\{\varphi\in \mathcal{SH}_m(\Omega)\;/\; \varphi\leq \min(\omega^{g_j},H)\},$$ this implies that $v_j\in N_m(H)$ and $H_m(v_j)=H_m(\omega^{g_j})$. Now, since 
$$\mathfrak{B}(H_m(\psi_j),\min(\omega^{g_j}, H))=\mathfrak{B}(H_m(\psi_j),v_j)$$
Then by $(b)$ in Proposition \ref{37}, 
$u_j \in \mathcal{N}_m(H)\; \hbox{and}\; H_m(u_j)=H_m(\psi_j)+H_m(\omega^{g_j}).$
By letting  $j\longrightarrow+\infty$ we obtain the desired result. Note that $\omega+H\leq u_j\leq H$, so if $\omega\in\mathcal{N}_m$ then $u\in\mathcal{N}_m(\Omega).$
\end{proof}
\par \textbf{Acknowledgement.} I am grateful to my supervisors Hichame Amal and Said Asserda for support, suggestions and encouragement.  I also would like to thank  Ahmed Zeriahi for very useful discussions at Ibn tofail university and during my visit to Institut de math\'{e}matiques de Toulouse.

\par \textbf{ Ibn Tofail University, Faculty of Sciences, Kenitra, Morocco.}
\par \textbf{ E-mail adress: ayoub.el-gasmi@uit.ac.ma }
\end{document}